\documentclass[10pt,leqno]{amsart}
\usepackage{srcltx}

\usepackage{amsmath,amssymb,cases,color}
\usepackage{hyperref}
\usepackage{bm}
\usepackage[capitalise]{cleveref}
\usepackage{vmargin}
\usepackage{dsfont}
\usepackage{tikz}

\usepackage{amsmath, amssymb,amscd, }
\usepackage{amsfonts}
\usepackage{mathrsfs}
\usepackage{graphicx}

 \usepackage{upref}
\hypersetup{linkcolor=blue, colorlinks=true,citecolor = red}

\newtheorem{theorem}{Theorem}[section]
\newtheorem{corollary}{Corollary}
\newtheorem*{main}{Main Theorem}
\newtheorem{lemma}[theorem]{Lemma}
\newtheorem{proposition}{Proposition}

\theoremstyle{definition}
\newtheorem{definition}[theorem]{Definition}
\newtheorem{remark}{Remark}

\newcommand{\R}{\mathbb{R}}
\newcommand{\bu}{\mathbf{u}}
\newcommand{\by}{\mathbf{y}}
\newcommand{\PHI}{\mathbf{\Phi}}
\newcommand{\bz}{\mathbf{z}}
\newcommand{\bg}{\mathbf{g}}
\newcommand{\bw}{\mathbf{w}}

\newcommand{\V}{\mathbf{V}}
\newcommand{\wt}{\widetilde}
\newcommand{\mc}{\mathcal}

\title[] 
      {Wellposedness of Boussinesq system}

\author[ Arnab Roy]{}


 \email{royarnab244@gmail.com}

\begin{document}
\maketitle

\centerline{\scshape Arnab Roy}
\medskip
{\footnotesize
 \centerline{Technische Universit\"{a}t Darmstadt, }
   \centerline{ Schlo\ss{}gartenstra{\ss}e 7, 64289 Darmstadt, Germany.}
} 

\medskip

%


\begin{abstract}
In this paper, we consider the viscous, incompressible, nonlinear Boussinesq system in two and three spatial dimension. We study the existence and regularity of solutions to the Boussinesq system with nonhomogeneous  boundary conditions for which the normal component for the velocity is not necessary equal to zero. We establish the existence of global weak solutions to the three-dimensional instationary nonlinear Boussinesq system but without any smallness assumption on the initial and boundary conditions. 
\end{abstract}

\section{Introduction}
 In this paper we are mainly interested in the following initial boundary value problem for non-stationary nonlinear Boussinesq system that describes the flow of a viscous incompressible fluid in $\Omega \subset \R^{N},\, N=2,3$ with smooth boundary $\Gamma$, subject to convective heat transfer : 
 \begin{align}\label{prothomeq}
\left\{ \begin{aligned}  
 & \frac{\partial \mathbf{z}}{\partial t} - \nu\Delta \mathbf{z} + (\mathbf{z}.\nabla)\mathbf{z} + \nabla p= \bm {\beta}\theta + \mathbf{f}_{1} \mbox{ in } Q=\Omega \times (0,T), \\ & \mbox{div }\mathbf{z}=0 \mbox{ in } Q=\Omega \times (0,T);\\ &\mathbf{z}=\mathbf{g} \mbox{ on } \Sigma=\Gamma \times (0,T),\quad \mathbf{z}(0)=\mathbf{z}_{0}(x) \mbox{ in } \Omega, \\& \frac{\partial \theta}{\partial t} - \mu\Delta\theta + (\mathbf{z}.\nabla)\theta  = f_{2} \mbox{ in }\mbox { Q },  \\& \frac{\partial\theta}{\partial n} = h \mbox{ on } \Sigma ,\quad \theta(0)=\theta_{0} \mbox{ in } \Omega.
\end{aligned}
\right.
\end{align}
Here $\bz=(z_{1},z_{2},....,z_{n})$ is the fluid velocity, $p$ denotes the pressure and $\theta$ is the temperature. Here $i$-th component of $(\bz.\nabla)\bz$ is $\sum_{j=1}^{n}z_{j}\frac{\partial z_{i}}{\partial x_{j}}$ and $\frac{\partial \theta}{\partial n}$ denotes the outer normal derivative of $\theta$ at $x$ on $\Gamma$. In equation \eqref{prothomeq}, $\bm{\beta}$ is gravitational vector function, $\nu>0$ is kinematic viscosity and $\mu>0$ is the thermal diffusivity, $\mathbf{f}_{1}$ is an external force and $f_{2}$ is a heat source applied to the fluid. Also, the boundary conditions $\bg(x,t)$ and $h(x,t)$ are functions defined on $\Gamma \times (0,T)$. In this Boussinesq approximation, the fluid is treated as incompressible when formulating the Navier-Stokes mass and momentum conservation equations and here the effect of temperature change is taken into account in the buoyancy term $\bm{\beta}\theta$ which drives convection.\\

 In this paper, our aim is to prove the existence of  the weak solution for system \eqref{prothomeq} and analyse its regularity. In \cite{Morimoto}, the author also considers the same non-homogeneous problem \eqref{prothomeq} in two and three dimension but with very smooth boundary data. Motivated by boundary control problems, here our main interest is in investigating the case when boundary data are not so regular. In that case, we cannot apply the extension method as in \cite{Morimoto}. We are also interested in finding a sufficient condition on $(\bg,h)$ so that a weak solution to equation \eqref{prothomeq} exists.\\

We denote the outward unit normal to the boundary $\Gamma$ by $\mathbf{n}$. When the normal component of $\mathbf{g}$ is zero, i.e, in the case of $\bg.\mathbf{n}=0$, the authors in \cite{GRUBB} and \cite{GRUBBS} studied existence and regularity results of Navier-Stokes equation by pseudo-differential techniques. We can adopt similar analysis in the case of equation \eqref{prothomeq} if $\bg.\mathbf{n}=0$. But for engineering and other practical applications (e.g \cite{ENG}), the interesting case to study is when $\bg.\mathbf{n}\neq 0$. To overcome this difficulty we will follow a similar approach as described in \cite{RAMO}. In \cite{RAMO} the author writes Stokes and Oseen equations in the form of a system of two operator equations and obtains the optimal regularity results. The first one is an evolution equation satisfied by the projection of the solution on the Stokes space - the space of divergence free vector fields with the normal trace equal to zero. The second one is a quasi-stationary equation satisfied by the projection of the solution on the orthogonal complement of the Stokes space. In our case, we have to deal with coupled system \eqref{prothomeq}. \\

We explore here the existence and regularity of solutions to linearized and nonlinear Boussinesq system with nonhomogeneous boundary conditions. The novelty of our work is that we consider boundary conditions with low regularity and the Dirichlet boundary condition for velocity for which the normal component is not equal to zero and Neumann boundary condition for temperature, that makes the problem more interesting and challenging. We also establish the existence of weak solution to the three dimensional nonlinear Boussinesq system \eqref{prothomeq} without any smallness assumption to initial condition. \\

Our main result concerning the existence and regularity of solution to system \eqref{prothomeq} is the following:
\begin{main}\label{mainthm}
Let $(P\bz(0),\theta(0))\in \V_{n}^{0}(\Omega)\times L^{2}(\Omega)$ and  $\bg\in \V^{s,s}(\Sigma)$, \\$h \in L^{2}(0,T;(H^{1-s}(\Gamma))') \cap H^{s}(0,T;(H^{1}(\Gamma))')$ with $\frac{1}{2}< s < 1$, $\mathbf{f}_{1} \in L^{2}(0,T; \V^{-1}(\Omega))$,\\ $f_{2} \in L^{2}(0,T; (H^{1}(\Omega))')$.Then equation \eqref{prothomeq} admits at least one weak solution $(\bz, \theta)$ in $C_{w}([0,T];\V^{0}(\Omega)\times L^{2}(\Omega)) \cap L^{2}(0,T;\V^{1}(\Omega))\times L^{2}(0,T;H^{1}(\Omega))$ and $p \in \mathcal{D}'(0,T; L^{2}(\Omega))$,  where $C_{w}([0,T];\V^{0}(\Omega)\times L^{2}(\Omega))$ is the subspace of $L^{\infty}(0,T;\V^{0}(\Omega)\times L^{2}(\Omega))$ which are continuous from $[0,T]$ into $\V^{0}(\Omega)\times L^{2}(\Omega)$ equipped with its weak topology. 
\end{main}
Here $P$ is the usual L\'{e}ray projector and the notations $\V^{0}(\Omega),\,\V^{1}(\Omega),\,\V^{s,s}(\Sigma)$ are as in section \ref{marker}. The proof of this existence and regularity result rely on combination of techniques described in \cite{Morimoto} and \cite{RAMO}. Actually, we split the system into two parts-homogeneous and non-homogeneous boundary data. We split the system such that the part with homogeneous boundary data has nonlinear terms and the part with non homogeneous boundary condition has only linear terms. Then we study the part with homogeneous boundary data as described in \cite{Morimoto}. To analyse the non homogeneous part, we will follow a similar technique as in \cite{RAMO}, by applying projection to the first equation and then writing the coupled system in the form of an evolution equation and a quasi-stationary equation.\\
  
  The works in the literature which are most relevant to our present paper are \cite{RAMO},\cite{Cannon},\cite{Oeda},\cite{Morimoto}. In \cite{Cannon}, authors showed the existence of a unique local in time weak solution in $\mathbb{R}^{2} \times (0,T)$. They also prove a global existence theorem of weak solution for small initial data, if the exterior force field depends on the temperature (in suitable spaces). Recently, in the case of $\Omega=\R^{2}$, the result of global existence of smooth solutions to nonlinear Boussinesq system is generalized to the cases of 'partial viscosity' (i.e., either $\nu>$ 0 and $\mu=$ 0, or $\nu=$ 0 and $\mu>$ 0) by Hou-Li \cite{HOU} and Chae \cite{CHAE} independently for the case of smooth initial data but without smallness assumption.\\

\^{O}eda studied the moving boundary case with Dirichlet boundary condition in \cite{Oeda}. The author considered the time dependent domain $\hat{\Omega}=\cup_{0\leq t \leq T}\Omega(t)\times \{t\}$, where $\Omega(t)$ is a bounded set in $\mathbb{R}^{N}(N=2 \mbox{ or } 3)$. Under certain assumptions the existence of a weak solution on any interval $[0,T]$ and of a unique strong solution on a small time interval $[0,\tau_{0}]$ is proved.\\

 In \cite{Morimoto}, the author treats the initial value problem \eqref{prothomeq} on a bounded domain in $\mathbb{R}^{N}\,(N=2 \mbox{ or }3)$, with no-slip boundary conditions for the velocity and allows the temperature to be prescribed (as a function of space and time) on one part of the boundary, while the temperature flux is prescribed on the rest of the boundary. He defines an appropriate notion of weak solution, and proves existence and uniqueness (for $N=2$) results for such solutions. But for the definition of weak solution, the author uses an extension of boundary value $\bg$ to $\overline{\Omega}\times (0,T)$ and solve the system of equations corresponding to it. Assuming $\bg\in C^{1}(\overline{\Gamma}\times (0,T))$, i.e, sufficiently regular boundary data, the proof of existence is based on the construction of approximate solution by Galerkin method and passage to the limit using apriori estimates.\\

 The rest of the paper is organized as follows. Section 2 is dedicated to notations and general functional framework. We study the steady linearized Boussinesq system in Section 3. We define a weak solution to this system and establish existence and regularity results. We also introduce transposition solution in the case of not so regular boundary data. Section 4 is devoted to unsteady linearized Boussinesq system. At first we consider linearization around stationary state and then we study a linearization around an instationary state which is needed in section 5 to analyze the nonlinear Boussinesq system with nonhomogeneous boundary condition. In section 5, we prove the existence of weak solution to nonlinear Boussinesq system in 3D case. To the best of our knowledge, this existence and regularity results for three dimensional nonlinear Boussinesq system seems to be new. 

\section{Notation and general functional framework}\label{marker}

We will use boldface letters to denote functions with values in $\mathbb{R}^n$ and spaces of functions with values in $\mathbb{R}^n$. Here we denote by $\mathbf{L}^2(\Omega)$ the space $(L^2(\Omega))^n=L^2(\Omega;\mathbb{R}^n)$. Similarly the spaces $(H_{0}^1(\Omega))^n$ and $(H^{-1}(\Omega))^n$ are $\mathbf{H_{0}^1}(\Omega)$ and $\mathbf{H^{-1}}(\Omega)$.  We also introduce different spaces of divergence free functions and some corresponding trace spaces :
  \begin{align*}
   \mathcal{V}(\Omega)&=\{ \mathbf{u}\in \mathbf{C}_{c}^{\infty}(\Omega)|\mbox{ div}\, \mathbf{u}=0\mbox{ in } \Omega \} ,\\   
   \textbf{V}_{0}^{1}(\Omega)&= \{ \mathbf{u}\in \mathbf{H}_{0}^{1}(\Omega)|\mbox{ div}\, \mathbf{u}=0\mbox{ in } \Omega,\, \mathbf{u}=0 \mbox{ on }\Gamma  \} ,\\
  \textbf{V}^{2}(\Omega)&=\{ \mathbf{u}\in \mathbf{H}^2(\Omega)|\mbox{ div}\, \mathbf{u}=0 \mbox{ in } \Omega \} ,\\
   \textbf{V}_{n}^{0}(\Omega)&=\{ \mathbf{u}\in \mathbf{L}^2(\Omega)|\mbox{ div}\, \mathbf{u}=0 \mbox{ in } \Omega,\, \mathbf{u}.n=0 \mbox{ on } \Gamma\} ,\\
\mathbf{V}^{1}(\Gamma)&=\{\mathbf{u}\in \mathbf{H}^{1}(\Omega)| \int_{\Gamma}\mathbf{u.n}=0 \mbox{ on }\Gamma \} .   
\end{align*}
Similarly we can define $\mathbf{V}^{s}(\Omega),\mathbf{V}_{n}^{s}(\Omega),\mathbf{V}_{0}^{s}(\Omega),\mathbf{V}^{s}(\Gamma)$ for $s\geq 0$. For $s<0, \mathbf{V}^{s}(\Gamma)$ is the dual space of $\mathbf{V}^{-s}(\Gamma)$, with $\mathbf{V}^{0}(\Gamma)$ as pivot space. For notational convenience, we define the space  
\begin{align*}
\mathcal{H}^{s}(\Omega) = 
\begin{cases}
H^{s}(\Omega)/\R \quad &\mbox{if }s>0 ,\\
(H^{-s}(\Omega)/\R)' \quad &\mbox{if }s<0 .
\end{cases}
\end{align*}
We introduce here some spaces for time dependent functions :  For $s \geq 0$ and $\sigma \geq 0$, define
\begin{align*}
\V^{s, \sigma}(Q)&= H^{\sigma}(0,T; \V^{0}(\Omega)) \cap L^{2}(0,T; \V^{s}(\Omega)),\\
\V^{s, \sigma}(\Sigma)&= H^{\sigma}(0,T; \V^{0}(\Gamma)) \cap L^{2}(0,T; \V^{s}(\Gamma)).
\end{align*}
We also introduce 
\begin{align*}
W(0,T;X,Y)=\{\mathbf{u}\in L^{2}(0,T; X)|\, \mathbf{u}' \in L^{2}(0,T;Y) \},
\end{align*}
where $X,Y$ are two Banach spaces.

 Let us denote by P, the orthogonal projection operator from $\mathbf{L}^{2}(\Omega)$ onto $\mathbf{V}_{n}^{0}(\Omega)$( see \cite[Chap III, Theorem 1.1]{GALDI} for details). We recall that for all $s\in[0,2]$ $P \in \mathcal{L}(\mathbf{H}^{s}(\Omega), \mathbf{V}^{s}_{n}(\Omega))$ (see \cite[Chap III, Lemma 1.2]{GALDI}). The operator P can be extended to a bounded operator from $\mathbf{H}^{-1}(\Omega)$ to $\V^{-1}(\Omega)$, that we still denote by P.

Throughout this paper, the letter C denotes a positive constant that may change from line to line. When particular positive constants are required, we use $C_{0}, C_{1}$, etc. If H is a Hilbert space, we denote by $||.||_{H}$ its corresponding norm, by $H^{'}$ its dual space and by $\langle . , . \rangle_{H^{'},H}$ the $H^{'}-H$ duality pairing. For any two Hilbert spaces $H_{1}$ and $H_{2}$ we use the notation $H_{1} \hookrightarrow H_{2}$ to indicate that $H_{1}$ is continuously embedded into $H_{2}$.  We denote by $\mathcal{L}(H_{1},H_{2})$ the space of all bounded linear operators from $H_{1}$ to $H_{2}$ and we use the notation $\mathcal{L}(H)$ for $\mathcal{L}(H,H)$.

\bigskip


\section{ Steady linearized Boussinesq system }
We first recall the existence result for stationary Boussinesq system with homogeneous boundary conditions :
\begin{align*}
-\nu \Delta \mathbf{z}_{s}+(\mathbf{z}_{s}.\nabla)\mathbf{z}_{s}+\nabla q_{s} &= \mathbf{\beta}\theta_{s} + \mathbf{f}_{1} \quad \mbox{in }\Omega,\\ \mbox{div}\, \mathbf{z}_{s}=0 \quad \mbox{in  }\Omega,\\ -\mu\Delta \theta_{s} + (\mathbf{z}_{s}.\nabla)\theta_{s}=f_{2} \quad \mbox{in }\Omega, \\
\mathbf{z}_{s}=0\quad \mbox{on }\Gamma; \quad \frac{\partial \theta_{s}}{\partial \mathbf{n}}=0 \mbox{  on }\Gamma , 
\end{align*} 
where  $\bm{\beta}\in \mathbf{L}^{\infty}(\Omega)$ is a vector. Then from \cite[Proposition 2.3]{lee2000analysis}, we know that for $\mathbf{f}_{1} \in \mathbf{L}^{2}(\Omega)$ and $f_{2} \in L^{2}(\Omega)$, there exists $(\mathbf{z}_{s},\theta_{s},q_{s})\in \mathbf{V}_{0}^{1}(\Omega)\cap \mathbf{H}^{2}(\Omega)\times H^{2}(\Omega)\times L^{2}(\Omega)/\R \cap H^{1}(\Omega).$\\

For $\lambda_{0}>0$, let us consider steady linearised Boussinesq System (linearized around $(\bz_{s},\theta_{s})$) with homogeneous boundary condition :
\begin{equation}\label{eq:steadyboussinesq}
\left. \begin{aligned}
 & \lambda_{0}\mathbf{u} - \Delta \mathbf{u} + (\mathbf{u}.\nabla)\mathbf{z}_{s} + (\mathbf{z}_{s}.\nabla)\mathbf{u} + \nabla p= \bm {\beta}\phi+\mathbf{f}_{1} \mbox{ in } \Omega,  \\ & \mbox{div }\mathbf{u}=0 \mbox{ in }\Omega;\quad \mathbf{u}=0 \mbox{ on } \Gamma,  \\& \lambda_{0}\phi-\Delta\phi + \mathbf{z}_{s}.\nabla\phi +\mathbf{u}.\nabla\theta_{s} =f_{2} \mbox{ in }\Omega,  \\& \frac{\partial\phi}{\partial n} = 0 \mbox{ on } \Gamma.  
\end{aligned}
\right\}
\end{equation}

In this section, for notational convenience, we will use $\bz_{s}=\bz$ and $\theta_{s}=\theta$.  Without loss of generality we will assume that $\nu= \mu= 1$.
\begin{definition}\label{homogeneous}
$(\mathbf{u},\phi)\in \mathbf{V}_{0}^{1}(\Omega)\times H^{1}(\Omega)$ is a \textbf{weak solution} to equation \eqref{eq:steadyboussinesq} if for every $(\mathbf{\Phi},\psi)\in \mathbf{V}_{0}^{1}(\Omega)\times H^{1}(\Omega)$ and $(\mathbf{f}_{1},f_{2})\in \mathbf{L}^{2}(\Omega) \times L^{2}(\Omega)$ :
\begin{align*}
&\lambda_{0}\int_{\Omega} \bu.\mathbf{\Phi}+ \int_{\Omega}\nabla \mathbf{u}.\nabla \mathbf{\Phi}+\int_{\Omega}((\mathbf{u}.\nabla) \bz).\mathbf{\Phi}+\int_{\Omega}((\mathbf{z}.\nabla)\mathbf{u}).\mathbf{\Phi}-\int_{\Omega} \bm{\beta}\phi.\mathbf{\Phi} \notag\\&+\lambda_{0}\int_{\Omega}\phi\psi+\int_{\Omega}\nabla \phi.\nabla \psi +\int_{\Omega}(\mathbf{z}.\nabla \phi)\psi +\int_{\Omega} (\mathbf{u}.\nabla \theta)\psi=\int_{\Omega} \mathbf{f}_{1}.\mathbf{\Phi}+ \int_{\Omega}f_{2}\psi.
\end{align*}
\end{definition}
We have the following existence theorem for weak solution :
\begin{theorem}\label{existenceboussi}
Let $(\bz,\theta)\in \V^{1}(\Omega)\times H^{1}(\Omega)$. Then there exists large $\lambda_{0}>0$ for which there exists unique weak solution $(\mathbf{u},p,\phi)\in \mathbf{V}_{0}^{1}(\Omega)\times L^{2}(\Omega)/ \R\times H^{1}(\Omega)$ to equation \eqref{eq:steadyboussinesq} for all $(\mathbf{f}_{1},f_{2})\in \mathbf{L}^{2}(\Omega) \times L^{2}(\Omega)$. 
\end{theorem}
\begin{proof}
Define the bilinear form on $\mathbf{V}_{0}^{1}(\Omega)\times H^{1}(\Omega)$  : 
\begin{align*}
a((\mathbf{u},\phi),(\mathbf{\Phi},\psi))=&\lambda_{0}\int_{\Omega}\mathbf{u.\Phi}+ \int_{\Omega}\nabla \mathbf{u}.\nabla \mathbf{\Phi}+\int_{\Omega}((\mathbf{u}.\nabla) \bz).\mathbf{\Phi}+\int_{\Omega}((\mathbf{z}.\nabla)\mathbf{u}).\mathbf{\Phi}-\int_{\Omega} \bm{\beta}\phi.\mathbf{\Phi} \notag\\&+\lambda_{0}\int_{\Omega}\phi\psi+\int_{\Omega}\nabla \phi.\nabla \psi +\int_{\Omega}(\mathbf{z}.\nabla \phi)\psi + \int_{\Omega} (\mathbf{u}.\nabla \theta)\psi
\end{align*}
This bilinear form is continuous on $\mathbf{V}_{0}^{1}(\Omega)\times H^{1}(\Omega)$. We want to show that it is coercive if $\lambda_{0}>0$ is large. For that we have to consider 
\begin{align*}
a((\mathbf{u},\phi),(\mathbf{u},\phi))=&\lambda_{0}\int_{\Omega}|\mathbf{u}|^{2}+ \int_{\Omega}|\nabla \mathbf{u}|^{2}+\int_{\Omega}((\mathbf{u}.\nabla)\mathbf{z}).\bu+\int_{\Omega}((\mathbf{z}.\nabla)\mathbf{u}).\mathbf{u}-\int_{\Omega} \bm{\beta}\phi.\mathbf{u} \notag\\&+\lambda_{0}\int_{\Omega}\phi^{2}+\int_{\Omega}|\nabla \phi|^{2} +\int_{\Omega}(\mathbf{z}.\nabla \phi)\phi + \int_{\Omega} (\mathbf{u}.\nabla \theta)\phi .
\end{align*} 
Now we will estimate term by term. Let us start with the third term and by H$\ddot{o}$lder's inequality :
\begin{align*}
\int_{\Omega}|((\mathbf{u}.\nabla)\mathbf{z}).\bu| \leq ||\mathbf{u}||^{2}_{\mathbf{L}^{4}(\Omega)}||\mathbf{z}||_{H^{1}(\Omega)}.
\end{align*}
Now observe that :
\begin{align*}
\int_{\Omega}((\mathbf{z}.\nabla)\mathbf{u}).\mathbf{\Phi}=-\int_{\Omega}((\mathbf{z}.\nabla)\mathbf{\Phi}).\mathbf{u}\,\, .
\end{align*}
Hence, we get  : $\displaystyle \int_{\Omega}((\mathbf{z}.\nabla)\mathbf{u}).\mathbf{u}=0$. Due to the similar reasoning, we also have :
\begin{align*}
\int_{\Omega}(\mathbf{z}.\nabla \phi)\phi=0 .
\end{align*}
Also by H$\ddot{o}$lder's inequality we have :
\begin{align*}
\int_{\Omega} \mathbf{\beta}\phi.\mathbf{u}&\leq ||\bm{\beta}||_{\infty}(\int_{\Omega}\phi^{2})^{1/2}(\int_{\Omega}|\mathbf{u}|^{2})^{1/2}\\& \leq \frac{||\bm{\beta}||_{\infty}}{2}(\int_{\Omega} \phi^{2}+\int_{\Omega}|\mathbf{u}|^{2}).
\end{align*}

 By using the fact that $H^{1}(\Omega)\hookrightarrow L^{4}(\Omega)$, the other coupling term gives : 
\begin{align*}
\displaystyle \int_{\Omega} |(\mathbf{u}.\nabla \theta)\phi | &\leq ||\theta||_{H^{1}(\Omega)}||\phi||_{L^{4}(\Omega)}||\mathbf{u}||_{\mathbf{L}^{4}(\Omega)}\\& \leq ||\theta||_{H^{1}(\Omega)}||\phi||_{H^{1}(\Omega)}||\mathbf{u}||_{\mathbf{L}^{4}(\Omega)} .
\end{align*}
So, we obtain :
\begin{align*}
a((\mathbf{u},\phi),(\mathbf{u},\phi)) \geq\, &\lambda_{0}\int_{\Omega}|\mathbf{u}|^{2}+ \int_{\Omega}|\nabla \mathbf{u}|^{2}-||\mathbf{u}||^{2}_{\mathbf{L}^{4}(\Omega)}||\mathbf{z}||_{\mathbf{H}^{1}(\Omega)}-\frac{||\bm{\beta}||_{\infty}}{2}(\int_{\Omega} \phi^{2}+\int_{\Omega}|\mathbf{u}|^{2})\\&+\lambda_{0}\int_{\Omega}\phi^{2}+\int_{\Omega}|\nabla \phi|^{2}-||\theta||_{H^{1}(\Omega)}||\phi||_{H^{1}(\Omega)}||\mathbf{u}||_{\mathbf{L}^{4}(\Omega)} .
\end{align*}

Now by using  $\V^{1}_{0}(\Omega)\hookrightarrow \mathbf{L}^{4}(\Omega)$, we will have :
\begin{align*}
||\mathbf{u}||^{2}_{\mathbf{L}^{4}(\Omega)}||\mathbf{z}||_{\mathbf{H}^{1}(\Omega)} &\leq C||\mathbf{u}||^{2}_{\mathbf{V}^{1}_{0}(\Omega)}||\mathbf{z}||_{\mathbf{H}^{1}(\Omega)}\\ & =C||\mathbf{u}||^{7/4}_{\mathbf{V}^{1}_{0}(\Omega)}||\bu||^{1/4}_{\V^{1}_{0}(\Omega)}||\mathbf{z}||_{\mathbf{H}^{1}(\Omega)} .
\end{align*}
Thus we will get :
\begin{align*}
-||\mathbf{u}||^{2}_{\mathbf{L}^{4}(\Omega)}||\mathbf{z}||_{\mathbf{H}^{1}(\Omega)}&\geq -C||\bu||^{7/4}_{\V^{1}_{0}(\Omega)}\big(||\bz||_{\mathbf{H}^{1}(\Omega)}||\bu||^{1/4}_{\V_{n}^{0}(\Omega)}\big)\\& \geq -\frac{1}{4}||\bu||^{2}_{\V^{1}_{0}(\Omega)}-\frac{7^{7}\times C^{8}}{8 \times 2^{7}}||\bz||^{8}_{\mathbf{H}^{1}(\Omega)}||\bu||^{2}_{\V^{0}_{n}(\Omega)},
\end{align*}
where the last inequality will come from Young's inequality.
Similarly, we can estimate :
\begin{align*}
||\theta||_{H^{1}(\Omega)}||\mathbf{u}||_{\mathbf{L}^{4}(\Omega)}||\phi||_{H^{1}(\Omega)}&\leq \frac{1}{2}\Big(||\theta||^{2}_{H^{1}(\Omega)}||\bu||^{2}_{L^{4}(\Omega)}+||\phi||^{2}_{H^{1}(\Omega)}\Big)\\&=\frac{1}{2}||\theta||^{2}_{H^{1}(\Omega)}||\bu||^{1/4}_{L^{4}(\Omega)}||\bu||^{7/4}_{L^{4}(\Omega)}+\frac{1}{2}||\phi||^{2}_{H^{1}(\Omega)}
\end{align*}
Hence,
\begin{align*}
-||\theta||_{H^{1}(\Omega)}||\phi||_{H^{1}(\Omega)}||\mathbf{u}||_{\mathbf{L}^{4}(\Omega)}&\geq-\frac{1}{2}||\phi||^{2}_{H^{1}(\Omega)}-C_{1} ||\bu||^{7/4}_{\V^{1}_{0}(\Omega)}||\theta||^{2}_{H^{1}(\Omega)}||\bu||^{1/4}_{\V_{n}^{0}(\Omega)}\\&\geq -\frac{1}{2}||\phi||^{2}_{H^{1}(\Omega)} -\frac{1}{4}||\bu||^{2}_{\V^{1}_{0}(\Omega)}-\frac{7^{7}\times C_{1}^{8}}{8 \times 2^{7}}||\theta||^{16}_{H^{1}(\Omega)}||\bu||^{2}_{\V^{0}_{n}(\Omega)}
\end{align*}
Thus we will get :
\begin{align*}
a((\mathbf{u},\phi),(\mathbf{u},\phi)) \geq &\lambda_{0}\int_{\Omega}|\mathbf{u}|^{2}+ \int_{\Omega}|\nabla \mathbf{u}|^{2} -\frac{1}{4}||\bu||^{2}_{\V^{1}_{0}(\Omega)}-\frac{7^{7}\times C^{8}}{8 \times 2^{7}}||\bz||^{8}_{\mathbf{H}^{1}(\Omega)}||\bu||^{2}_{\V^{0}_{n}(\Omega)}\\ & -\frac{||\bm{\beta}||_{\infty}}{2}(\int_{\Omega} \phi^{2}+\int_{\Omega}|\mathbf{u}|^{2})+\lambda_{0}\int_{\Omega}\phi^{2}+\int_{\Omega}|\nabla \phi|^{2}-\frac{1}{2}||\phi||^{2}_{H^{1}(\Omega)}\\& -\frac{1}{4}||\bu||^{2}_{\V^{1}_{0}(\Omega)}-\frac{7^{7}\times C_{1}^{8}}{8 \times 2^{7}}||\theta||^{16}_{H^{1}(\Omega)}||\bu||^{2}_{\V^{0}_{n}(\Omega)}.
\end{align*}
Now it is sufficient to choose  $\lambda_{0}\geq \max\{1+\frac{7^{7}\times C^{8}}{8 \times 2^{7}}||\bz||^{8}_{\mathbf{H}^{1}(\Omega)}+\frac{||\bm{\beta}||_{\infty}}{2},\, 1+\frac{7^{7}\times C_{1}^{8}}{8 \times 2^{7}}||\theta||^{16}_{H^{1}(\Omega)}+\frac{||\bm{\beta}||_{\infty}}{2}\}$,  such that following holds:
\begin{align*}
a((\mathbf{u},\phi),(\mathbf{u},\phi))\geq \frac{1}{2}\Big(||\mathbf{u}||_{\mathbf{V}^{1}_{0}(\Omega)}^{2}+||\phi||_{H^{1}(\Omega)}^{2}\Big).
\end{align*}   
So by Lax-Milgram theorem we have existence of the weak solution $(\mathbf{u},\phi)\in \mathbf{V}_{0}^{1}(\Omega)\times H^{1}(\Omega)$ to equation \eqref{eq:steadyboussinesq} for all $(\mathbf{f}_{1},f_{2})\in \mathbf{L}^{2}(\Omega) \times L^{2}(\Omega)$  . Then by applying de Rham theorem, we can recover pressure $p \in L^{2}(\Omega)/ \R$ as in the steady Stokes system in \cite{GALDI}.  
\end{proof}
\begin{theorem}\label{steadyRegularity}
For all $(\mathbf{f}_{1},f_{2})\in \mathbf{L}^{2}(\Omega) \times L^{2}(\Omega)$, the weak solution $(\mathbf{u},p,\phi)$ for system \eqref{eq:steadyboussinesq} given by theorem \eqref{existenceboussi} belongs to $\mathbf{V}^{2}(\Omega)\times \mathcal{H}^{1}(\Omega) \times H^{2}(\Omega)$.
\end{theorem}
\begin{proof}
We already know the existence of the solution $(\mathbf{u}, p, \phi)\in \mathbf{V}^{1}_{0}(\Omega)\times L^{2}(\Omega)/\R \times H^{1}(\Omega)$ from theorem \ref{existenceboussi}. So, consider 
\begin{align*}
\lambda_{0}\mathbf{u} - \Delta \mathbf{u} + (\mathbf{u}.\nabla)\mathbf{z} + (\mathbf{z}.\nabla)\mathbf{u} + \nabla p= \bm{\beta}\phi+\mathbf{f}_{1} \mbox{ in } \Omega
\end{align*}
Here as $(\bm{\beta}\phi+\mathbf{f}_{1})\in \mathbf{L}^{2}(\Omega)$, using the regularity result \cite[Appendix B, lemma B.1]{RAMO} we can conclude that $(\mathbf{u},p)\in \mathbf{V}^{2}(\Omega)\times (\mathcal{H}^{1}(\Omega)/\mathbb{R})$. Then consider the equation 
\begin{align*}
\lambda_{0}\phi- \Delta\phi + \mathbf{z}.\nabla\phi =f_{2}-\mathbf{u}.\nabla\theta\quad \mbox{ in   }\Omega .
\end{align*}
Here $\bu \in \mathbf{H}^{2}(\Omega) \hookrightarrow \mathbf{L}^{\infty}(\Omega)$ for two and three dimensions. So, $\bu.\nabla \theta \in L^{2}(\Omega)$. By elliptic regularity we can conclude that $\phi \in H^{2}(\Omega)$. 
\end{proof}
Now onwards we fix $\lambda_{0}$ as in theorem \ref{existenceboussi}. We consider linearised Boussinesq system with non homogeneous boundary data :
  \begin{equation} \label{eq:steadyboussinesq1}
\left. \begin{aligned}
 & \lambda_{0}\mathbf{u} - \Delta \mathbf{u} + (\mathbf{u}.\nabla)\mathbf{z} + (\mathbf{z}.\nabla)\mathbf{u} + \nabla p= \bm{\beta}\phi+\mathbf{f}_{1} \mbox{ in } \Omega,  \\ & \mbox{div }\mathbf{u}=0 \mbox{ in }\Omega;\quad \mathbf{u}=\mathbf{g} \mbox{ on } \Gamma , \\& \lambda_{0}\phi- \Delta\phi + \mathbf{z}.\nabla\phi +\mathbf{u}.\nabla\theta =f_{2} \mbox{ in }\Omega , \\& \frac{\partial\phi}{\partial n} = h \mbox{ on } \Gamma .
 \end{aligned}
\right\} 
\end{equation}
We first prove the existence of a unique solution when $(\bg,h)\in \V^{3/2}(\Gamma)\times H^{1/2}(\Gamma)$.
If $\mathbf{g}\in \mathbf{V}^{3/2}(\Gamma)$, we can define the Dirichlet operator as  $D_{z}\mathbf{g}=(\mathbf{w}, \pi)$ using the solution to equation:
 \begin{align}\label{wpi}
 &\lambda_{0}\mathbf{w}-\Delta\mathbf{w}+(\mathbf{w}.\nabla)\mathbf{z}+(\mathbf{z}.\nabla)\mathbf{w}+\nabla{\pi}=0 \mbox{ in }\Omega ,\notag\\& \mbox{ div}\,\mathbf{w}=0\mbox{ in }\Omega ,\quad \mathbf{w}=\mathbf{g}\mbox{ on }\Gamma.
 \end{align} 
Then $D_{z}$ is  linear and continuous from $\mathbf{V}^{3/2}(\Gamma)$ to $\mathbf{V}^{2}(\Omega) \times \mathcal{H}^{1}(\Omega)/\R$. For details see \cite[Appendix B, Corollary B.1]{RAMO}. \\ Then consider $\widetilde{\mathbf{w}}=\mathbf{u-w},\widetilde{p}=p-\pi$. Here $(\widetilde{\mathbf{w}},\widetilde{p})$ satisfies : 
 \begin{align*} 
 & \lambda_{0}\widetilde{\mathbf{w}} - \Delta \widetilde{\mathbf{w}} + (\widetilde{\mathbf{w}}.\nabla)\mathbf{z} + (\mathbf{z}.\nabla)\widetilde{\mathbf{w}} + \nabla \widetilde{p}= \bm {\beta}\phi+\mathbf{f}_{1} \mbox{ in } \Omega, \notag \\ & \mbox{div }\widetilde{\mathbf{w}}=0 \mbox{ in }\Omega;\quad \widetilde{\mathbf{w}}=0 \mbox{ on } \Gamma .\notag
 \end{align*}
 Now system \eqref{eq:steadyboussinesq1} can be reduced to the following system:
  \begin{equation}\label{newboussi} 
  \left. \begin{aligned}
 & \lambda_{0}\widetilde{\mathbf{w}} - \Delta \widetilde{\mathbf{w}} + (\widetilde{\mathbf{w}}.\nabla)\mathbf{z} + (\mathbf{z}.\nabla)\widetilde{\mathbf{w}} + \nabla \widetilde{p}= \bm {\beta}\phi+\mathbf{f}_{1} \mbox{ in } \Omega , \\ & \mbox{div }\widetilde{\mathbf{w}}=0 \mbox{ in }\Omega;\quad \widetilde{\mathbf{w}}=0 \mbox{ on } \Gamma , \\& \lambda_{0}\phi- \Delta\phi + \mathbf{z}.\nabla\phi +\widetilde{\mathbf{w}}.\nabla\theta =f_{2}-\mathbf{w}.\nabla \theta= \widetilde{f_{2}} \mbox{ in }\Omega , \\& \frac{\partial\phi}{\partial n} = h \mbox{ on } \Gamma .
 \end{aligned}
\right\} 
\end{equation}
For this system with inhomogeneous boundary condition, we define a weak solution in the same lines as in definition \ref{homogeneous}.
  \begin{definition}\label{inhomogeneous}
Let $h\in H^{1/2}(\Gamma)$, $(\bz,\theta) \in \V^{1}(\Omega) \times H^{1}(\Omega)$ and $(\mathbf{f}_{1},f_{2})\in \mathbf{L}^{2}(\Omega) \times L^{2}(\Omega)$. Then $(\widetilde{\mathbf{w}},\tau)\in \mathbf{V}_{0}^{1}(\Omega)\times H^{1}(\Omega)$ is a \textbf{weak solution} to equation \eqref{newboussi} if for every $(\mathbf{\Phi},\psi)\in \mathbf{V}_{0}^{1}(\Omega)\times H^{1}(\Omega)$,  the following holds: 
\begin{align*}
&\lambda_{0}\int_{\Omega}\widetilde{\mathbf{w}}.\mathbf{\Phi}+ \int_{\Omega}\nabla \widetilde{\mathbf{w}}.\nabla \mathbf{\Phi}+\int_{\Omega}((\widetilde{\mathbf{w}}.\nabla)\mathbf{z}+(\mathbf{z}.\nabla)\widetilde{\mathbf{w}}).\mathbf{\Phi}-\int_{\Omega} \mathbf{\beta}\tau.\mathbf{\Phi}\\&+\lambda_{0}\int_{\Omega}\phi \psi + \int_{\Omega}\nabla \phi.\nabla \psi + \int_{\Omega}(\mathbf{z}.\nabla \phi)\psi + \int_{\Omega} (\widetilde{\mathbf{w}}.\nabla \theta)\psi=\int_{\Omega} \mathbf{f}_{1}.\mathbf{\Phi}+\int_{\Omega}\widetilde{ f_{2}}\psi+\int_{\Gamma}h\psi 
\end{align*}
\end{definition}
\begin{theorem}\label{boussi2}
Let $(\bz, \theta) \in \V^{1}(\Omega) \times H^{1}(\Omega)$. For all $(\mathbf{f}_{1},f_{2})\in \mathbf{L}^{2}(\Omega) \times L^{2}(\Omega)$ and $(\mathbf{g},h)\in \mathbf{V}^{3/2}(\Gamma)\times H^{1/2}(\Gamma)$, there exists a weak solution $(\mathbf{u},p,\phi)\in \mathbf{V}^{1}(\Omega)\times L^{2}(\Omega) \times H^{1}(\Omega)$ to equation \eqref{eq:steadyboussinesq1}. Moreover, we have $(\mathbf{u},p,\phi)\in \mathbf{V}^{2}(\Omega)\times H^{1}(\Omega) \times H^{2}(\Omega)$. 
\end{theorem}
\begin{proof}
 Let $(\widetilde{\mathbf{w}}, \widetilde{p}, \phi)$ satisfy equation \eqref{newboussi}. As in the previous existence theorem, we can define a continuous, coercive bilinear form $a((\widetilde{\mathbf{w}},\phi),(\mathbf{\Phi},\psi))$ and continuous linear form $L((\mathbf{\Phi},\psi))$ on $\mathbf{V}_{0}^{1}(\Omega)\times H^{1}(\Omega)$.
By Lax-Milgram, we can conclude that $(\widetilde{\mathbf{w}},\phi)\in \mathbf{V}_{0}^{1}(\Omega)\times H^{1}(\Omega)$ is a weak solution of \eqref{newboussi}.  Hence by applying de Rham's theorem, we have $(\widetilde{\mathbf{w}}, \widetilde{p},\phi)\in\mathbf{V}^{1}_{0}(\Omega)\times L^{2}(\Omega)\times H^{1}(\Omega)$. 
  
  Now we know that $(\bu, p)= (\widetilde{\mathbf{w}}+ \mathbf{w}, \widetilde{p}+ \pi )$, where $(\mathbf{w}, \pi)$
  is the solution of \eqref{wpi} that belongs to $\V^{2}(\Omega) \times H^{1}(\Omega)/ \R$. So we can conclude that $(\mathbf{u},p,\phi)\in \mathbf{V}^{1}(\Omega)\times L^{2}(\Omega) \times H^{1}(\Omega)$.\\ 
  
  Now for proving regularity, we consider $(\mathbf{u},p)=(\mathbf{u}_{1}+\mathbf{u}_{2},p_{1}+p_{2})$ such that :
\begin{equation}\label{split1}
\left. \begin{aligned}
&\lambda_{0}\mathbf{u}_{1}-\Delta \mathbf{u}_{1}+\nabla p_{1}=0 \mbox{ in }\Omega,\, \mbox{ div}\,\mathbf{u}_{1}=0,\mbox{ in }\Omega,\,\quad \mathbf{u}_{1}=\mathbf{g} \mbox{ on }\Gamma \\& \lambda_{0}\mathbf{u}_{2} - \Delta \mathbf{u}_{2} + (\mathbf{u}_{2}.\nabla)\mathbf{z} + (\mathbf{z}.\nabla)\mathbf{u}_{2} + \nabla p_{2}=-(\mathbf{u}_{1}.\nabla)\mathbf{z}-(\mathbf{z}.\nabla)\mathbf{u}_{1}+ \mathbf {\beta}\phi+\mathbf{f}_{1} \mbox{ in } \Omega,  \\ & \mbox{div }\mathbf{u}_{2}=0 \mbox{ in }\Omega;\quad \mathbf{u}_{2}=0 \mbox{ on } \Gamma .
 \end{aligned}
\right\}
\end{equation}
 And $\phi=(\phi_{1}+\phi_{2})$ satisfies :
\begin{equation}\label{split2}
\left. \begin{aligned}
 &\lambda_{0}\phi_{1}-\Delta \phi_{1}+(\mathbf{z}.\nabla)\phi_{1}=0 \mbox{ in }\Omega,\,\quad \frac{\partial \phi_{1}}{\partial n}=h \mbox{ on }\Gamma,\\&  \lambda_{0}\phi_{2}- \Delta\phi_{2} + \mathbf{z}.\nabla\phi_{2} +\mathbf{u}_{2}.\nabla\theta =-\mathbf{u}_{1}.\nabla \theta +f_{2} \mbox{ in }\Omega, \\& \frac{\partial\phi_{2}}{\partial n} = 0 \mbox{ on } \Gamma. 
 \end{aligned}
\right\}
\end{equation}
Now the first equation of \eqref{split1} is the Stokes Problem and we know that $(\mathbf{u}_{1},p_{1})\in \mathbf{V}^{2}(\Omega)\times H^{1}(\Omega)$ if $\bg \in \V^{3/2}(\Gamma)$. Now the R.H.S of second equation in \eqref{split1} is in $\mathbf{L}^{2}(\Omega)$. Thus by \cite[Appendix B, lemma B.1]{RAMO}, we can conclude $(\mathbf{u}_{2},p_{2})\in \mathbf{V}^{2}(\Omega)\times H^{1}(\Omega)$. Hence, we can conclude that $(\mathbf{u},p)\in \mathbf{V}^{2}(\Omega)\times H^{1}(\Omega)$. 

Also from first equation of \eqref{split2}, we have $\phi_{1}\in H^{2}(\Omega)$ when $h \in H^{1/2}(\Gamma)$. Then consider the equation :
\begin{align*}
&  \lambda_{0}\phi_{2}- \Delta\phi_{2} + \mathbf{z}.\nabla\phi_{2} =-\mathbf{u}_{2}.\nabla\theta-\mathbf{u}_{1}.\nabla \theta +f_{2} \mbox{ in }\Omega \\& \frac{\partial\phi_{2}}{\partial n} = 0 \mbox{ on } \Gamma 
\end{align*}
 Now the R.H.S of the above equation is in $L^{2}(\Omega)$. Thus we have $\phi_{2}\in H^{2}(\Omega)$ .
\end{proof}
Now we consider the boundary data $(\mathbf{g},h)\in \mathbf{V}^{-1/2}(\Gamma)\times H^{-3/2}(\Gamma)$. In this case we can use transposition method to define a solution to equation \eqref{eq:steadyboussinesq1}. For simplicity let us assume that  $(\mathbf{f}_{1},f_{2})=(0,0)$.
\begin{definition}\label{steady transposition}
 Assume that $(\mathbf{g},h)\in \mathbf{V}^{-1/2}(\Gamma)\times H^{-3/2}(\Gamma).$ A function $(\mathbf{u},p,\phi)\in \mathbf{V}^{0}(\Omega)\times \mathcal{H}^{-1}(\Omega) \times L^{2}(\Omega)$ is a \textbf{transposition solution} to system \eqref{eq:steadyboussinesq1} if
 \begin{align}\label{trans1}
 \langle (\mathbf{u},\phi),(\mathbf{f}_{3},f_{4}) \rangle_{\mathbf{V}^{0}(\Omega)\times L^{2}(\Omega)}=\langle -\frac{\partial \mathbf{r}}{\partial \mathbf{n}}+\widetilde{\pi}\mathbf{n}-c(\widetilde{\pi})\mathbf{n}, \mathbf{g} \rangle_{\mathbf{V}^{1/2}(\Gamma),\mathbf{V}^{-1/2}(\Gamma)} + \langle  s, h \rangle_{H^{3/2}(\Gamma),H^{-3/2}(\Gamma)}
 \end{align}
 for every $(\mathbf{f}_{3},f_{4})\in \mathbf{V}^{0}(\Omega)\times L^{2}(\Omega)$, where $c(\widetilde{\pi})=\frac{1}{|\Gamma|}\int_{\Gamma}\widetilde{\pi}$ and $(\mathbf{r},s)$ is the solution to :
 \begin{equation}\label{rsstationary}
\left. \begin{aligned}
& \lambda_{0}\mathbf{r} - \Delta \mathbf{r} +(\nabla \mathbf{z})^{T}\mathbf{r} - (\mathbf{z}.\nabla)\mathbf{r} + \nabla \tilde{\pi}+(\nabla\theta)^{T} s= \mathbf{f}_{3} \mbox{ in } \Omega ,  \\ & \mbox{div }\mathbf{r}=0 \mbox{ in }\Omega;\quad \mathbf{r}=0 \mbox{ on } \Gamma , \\& \lambda_{0}s- \Delta s - \mathbf{z}.\nabla s -\bm{\beta.r} =f_{4} \mbox{ in }\Omega , \\& \frac{\partial s}{\partial n} = 0 \mbox{ on } \Gamma ,
\end{aligned}
\right\}
\end{equation}
and 
\begin{align*}
\langle p,k \rangle_{\mathcal{H}^{-1}(\Omega), \mathcal{H}^{1}(\Omega)}=\langle \frac{\partial \mathbf{r}_{1}}{\partial \mathbf{n}}- \pi_{1}\mathbf{n} + c(\pi_{1})\mathbf{n}, \mathbf{g} \rangle_{\mathbf{V}^{1/2}(\Gamma),\mathbf{V}^{-1/2}(\Gamma)}-\langle  s_{1},h \rangle_{H^{3/2}(\Gamma),H^{-3/2}(\Gamma)}
\end{align*}
for every $k\in \mathcal{H}^{1}(\Omega) $ with $\int_{\Omega}k=0$ and $(\mathbf{r}_{1},s_{1})$ is the solution to 
\begin{equation}\label{r1s1stationary}
\left. \begin{aligned}
& \lambda_{0}\mathbf{r}_{1} - \Delta \mathbf{r}_{1} +(\nabla \mathbf{z})^{T}\mathbf{r}_{1} - (\mathbf{z}.\nabla)\mathbf{r}_{1} + \nabla \pi_{1}+(\nabla\theta)^{T} s_{1}= 0 \mbox{ in } \Omega  \\ & \mbox{div }\mathbf{r}_{1}=k \mbox{ in }\Omega;\quad \mathbf{r}_{1}=0 \mbox{ on } \Gamma  \\& \lambda_{0}s_{1}- \Delta s_{1} - \mathbf{z}.\nabla s_{1} -\bm{\beta}.\mathbf{r}_{1} =0 \mbox{ in }\Omega  \\& \frac{\partial s_{1}}{\partial n} = 0 \mbox{ on } \Gamma 
\end{aligned}
\right\}
\end{equation}
 \end{definition}
\begin{theorem}\label{transposition1}
 For all $(\mathbf{g},h)\in \mathbf{V}^{-1/2}(\Gamma)\times H^{-3/2}(\Gamma)$, system \eqref{eq:steadyboussinesq1} admits a unique transposition solution $(\mathbf{u},p,\phi)\in \mathbf{V}^{0}(\Omega)\times \mathcal{H}^{-1}(\Omega) \times L^{2}(\Omega)$ in the sense of above definition.
 \end{theorem}
 \begin{proof}
 As $(\mathbf{f}_{3},f_{4})\in \mathbf{V}^{0}(\Omega)\times L^{2}(\Omega)$, by theorem \ref{steadyRegularity} we have $(\mathbf{r}, \tilde{\pi}, s)\in \mathbf{V}^{2}(\Omega)\times \mathcal{H}^{1}(\Omega) \times H^{2}(\Omega)$. Let us define the operator $\Lambda$ from $\mathbf{V}^{0}(\Omega)\times L^{2}(\Omega)$ into $\mathbf{V}^{1/2}(\Gamma)\times H^{3/2}(\Gamma)$ by : 
\begin{align*} 
 \Lambda(\mathbf{f}_{3},f_{4})=\Big(-\frac{\partial \mathbf{r}}{\partial \mathbf{n}}+\widetilde{\pi}\mathbf{n}-c(\widetilde{\pi})\mathbf{n},s\Big),
 \end{align*}
 where $(\mathbf{r},s)$ is solution to \eqref{rsstationary}. Now we have :
 \begin{align*}
 || \Lambda(\mathbf{f}_{3},f_{4})||_{\mathbf{V}^{1/2}(\Gamma)\times H^{3/2}(\Gamma)} &= || \Big(-\frac{\partial \mathbf{r}}{\partial \mathbf{n}}+\widetilde{\pi}\mathbf{n}-c(\widetilde{\pi})\mathbf{n},s\Big)||_{\mathbf{V}^{1/2}(\Gamma)\times H^{3/2}(\Gamma)} \\& \leq C( || \mathbf{r}||_{\V^{2}(\Omega)} + || \tilde{\pi} ||_{H^{1}(\Omega)} + ||s||_{H^{2}(\Omega)}) \\ & \leq C(||\mathbf{f}_{3}||_{\V^{0}(\Omega)} + ||f_{4}||_{L^{2}(\Omega)}) \\ & = C || (\mathbf{f}_{3}, f_{4})||_{\V^{0}(\Omega)\times L^{2}(\Omega)}.
 \end{align*}
 Then it follows that $\Lambda$ is a bounded, linear operator. Then equation \eqref{trans1} can be rewritten as : 
 \begin{align*}
 \langle (\mathbf{u},\phi),(\mathbf{f}_{3},f_{4}) \rangle_{\mathbf{V}^{0}(\Omega)\times L^{2}(\Omega)}&=\langle \Lambda(\mathbf{f}_{3},f_{4}),(\mathbf{g},h)\rangle_{\mathbf{V}^{1/2}(\Gamma)\times H^{3/2}(\Gamma),\mathbf{V}^{-1/2}(\Gamma)\times H^{-3/2}(\Gamma)}\\&=\langle (\mathbf{f}_{3},f_{4}),\Lambda^{*}(\mathbf{g},h) \rangle_{\mathbf{V}^{0}(\Omega)\times L^{2}(\Omega)}
 \end{align*}
where $\Lambda^{*}$ is the adjoint of $\Lambda$ and $\Lambda^{*}\in \mathcal{L}(\mathbf{V}^{-1/2}(\Gamma)\times H^{-3/2}(\Gamma),\mathbf{V}^{0}(\Omega)\times L^{2}(\Omega))$. The function $(\mathbf{u},\phi)=\Lambda^{*}(\mathbf{g},h)$ is clearly a solution to system \eqref{eq:steadyboussinesq1} in the sense of  definition \ref{steady transposition}.\\

Now we want to prove the existence of pressure $p$. For $k \in \mathcal{H}^{1}(\Omega)$ with $\int_{\Omega} k = 0$, the solution $(\mathbf{r}_{1},s_{1})$ to equation \eqref{r1s1stationary} belongs to $\V^{2}(\Omega) \times H^{2}(\Omega)$. The operator $\Lambda_{1}$ from $\mathcal{H}^{1}(\Omega)$ to $\V^{\frac{1}{2}}(\Gamma))\times H^{\frac{3}{2}}(\Gamma)$ 
can be defined as :
\begin{align*}
\Lambda_{1}(k)= \Big( -\frac{\partial \mathbf{r}_{1}}{\partial \mathbf{n}}+\pi_{1}\mathbf{n}-c(\pi_{1})\mathbf{n},\, s_{1} \Big),
\end{align*}
where $(\mathbf{r}_{1},s_{1})$ is the solution to equation \eqref{r1s1stationary}. It can be proved similarly that the function 
\begin{align*}
p= \Lambda_{1}^{*}\begin{pmatrix}
\mathbf{g} \\ h 
\end{pmatrix} \in \mathcal{H}^{-1}(\Omega)
\end{align*}
is a solution to equation \eqref{eq:steadyboussinesq1} in the sense of definition \ref{steady transposition}.\\

To prove uniqueness, let $(\bu,\phi)$ be a solution corresponding to $(\bg,h)=(0,0).$ Thus from equation \eqref{trans1}, for all $(\mathbf{f}_{3}, f_{4}) \in \V^{0}(\Omega) \times L^{2}(\Omega)$,  we have :
\begin{align*}
 \langle (\mathbf{u},\phi),(\mathbf{f}_{3},f_{4}) \rangle_{\mathbf{V}^{0}(\Omega)\times L^{2}(\Omega)}=0 .
\end{align*}
Hence, we will get $(\bu,\phi)=(0,0)$.
\end{proof}

Our next aim is to define lifting operator $L_{\bz}$ for the coupled system \eqref{eq:steadyboussinesq1}. We denote by $L_{\bz}(\mathbf{g},h)=(\mathbf{w}, p, \xi)$ and it satisfies the equation : 
 \begin{align*}
 & \lambda_{0}\mathbf{w} - \nu\Delta \mathbf{w} + (\mathbf{w}.\nabla)\mathbf{z} + (\mathbf{z}.\nabla)\mathbf{w} + \nabla p= \bm {\beta}\xi \mbox{ in } \Omega, \notag \\ & \mbox{div }\mathbf{w}=0 \mbox{ in }\Omega;\quad \mathbf{w}=\mathbf{g} \mbox{ on } \Gamma ,\notag \\& \lambda_{0}\xi- \mu\Delta\xi + \mathbf{z}.\nabla\xi +\mathbf{w}.\nabla\theta =0 \mbox{ in }\Omega ,\notag \\& \frac{\partial\xi}{\partial n} = h \mbox{ on } \Gamma .
 \end{align*}
 \begin{theorem}\label{liftboussinesq}
 The operator $L_{\bz}$ is linear and continuous from $\mathbf{V}^{s+1}(\Gamma)\times H^{s}(\Gamma)$ into $\mathbf{V}^{s+3/2}(\Omega)\times \mathcal{H}^{s+1/2}(\Omega)/\R \times H^{s+3/2}(\Omega)$ for all  $-3/2 \leq s \leq 1/2$.
 \end{theorem}
 \begin{proof}
 Due to theorem \ref{boussi2}, $L_{\bz}$ is continuous from $\mathbf{V}^{3/2}(\Gamma)\times H^{1/2}(\Gamma)$ to $\mathbf{V}^{2}(\Omega)\times H^{2}(\Omega).$ From theorem \ref{transposition1}, $L_{\bz}$ is also a bounded operator from $\mathbf{V}^{-1/2}(\Gamma)\times H^{-3/2}(\Gamma)$ into $\mathbf{V}^{0}(\Omega)\times (H^{1}(\Omega)/\R)' \times L^{2}(\Omega)$. Then our required result follows from interpolation theorem \cite[Chapter 1, theorem 9.6]{LM}. 
 \end{proof}
 \section{Unsteady linearized Boussinesq system}
 \subsection{Linearization around stationary state}

  In this section we study first the  linearized Boussinesq system around zero solution and later around nonzero solution with non homogeneous boundary conditions.

 \subsubsection{Linearization around zero solution}
 Here we are mainly interested in the following equation :
 \begin{equation}\label{aroundzero}
\left. \begin{aligned}
&\frac{\partial \mathbf{u}}{\partial t} -\Delta \mathbf{u} + \nabla p= \bm{\beta}\phi \mbox{ in } Q, \\ & \mbox{div }\mathbf{u}=0 \mbox{ in } Q;\quad \mathbf{u}=\mathbf{g} \mbox{ on } \Sigma,\\& \mathbf{u}(x,0)=\mathbf{u}_{0}(x) \mbox{ in }\Omega, \\&\frac{\partial \phi}{\partial t} - \Delta\phi  = 0 \mbox{ in } Q , \\
& \frac{\partial\phi}{\partial n} = h \mbox{ on } \Sigma,\quad \phi(0)=\phi_{0}  \mbox{ in } \Omega.
 \end{aligned}
\right\}
\end{equation}

If the boundary data is regular then there are some known existence and regularity results for system \eqref{prothomeq}. For instance, in \cite{Morimoto}, the author proved that if $\bg \in C^{1}(\Sigma)$ and $h \in L^{2}(\Sigma)$, then $(\bz, p, \theta)\in L^{\infty}(0,T; \V^{0}(\Omega)) \times L^{2}(0,T; L^{2}(\Omega))\times L^{\infty}( 0,T; L^{2}(\Omega))$. So, our interest lies in the case when the boundary data are not in such regular spaces.

We want to define a solution to equation \eqref{aroundzero} when the boundary data  $(\mathbf{g},h)$ belongs to $ L^{2}(0,T;\mathbf{V}^{-1/2}(\Gamma))\times L^{2}(0,T;H^{-3/2}(\Gamma))$. In this case, we are going to define the solution via transposition method.
\begin{definition}\label{sa}
Assume that $(\mathbf{g},h)\in L^{2}(0,T;\mathbf{V}^{-1/2}(\Gamma))\times L^{2}(0,T;H^{-3/2}(\Gamma))$ and $(\mathbf{u}_{0},\phi_{0})\in \mathbf{V}^{-1}(\Omega) \times (H^{1}(\Omega))^{'}$. A function $(\mathbf{u},p,\phi)\in L^{2}(0,T;\mathbf{V}^{0}(\Omega))\times L^{2}(0,T;\mathcal{H}^{-1}(\Omega)) \times \\ L^{2}(0,T;L^{2}(\Omega))$ is a solution to equation \eqref{aroundzero} in the sense of \textbf{transposition} iff 
\begin{align}\label{transposition}
\int_{Q} \mathbf{u.f}_{3}+ \int_{Q} \phi f_{4}=&\int_{0}^{T}\langle -\frac{\partial \mathbf{r}}{\partial \mathbf{n}}(t)+\widetilde{\pi}(t)\mathbf{n}-c(\widetilde{\pi})\mathbf{n}, \mathbf{g}(t) \rangle_{\mathbf{V}^{1/2}(\Gamma),\mathbf{V}^{-1/2}(\Gamma)} +\notag\\ & \langle  s(t), h(t) \rangle_{H^{3/2}(\Gamma),H^{-3/2}(\Gamma)} +\langle \mathbf{u}_{0},\mathbf{r}(0)\rangle_{\mathbf{H}^{-1}(\Omega),\mathbf{H}_{0}^{1}(\Omega)}+\notag\\& \langle \phi_{0},s(0) \rangle_{(H^{1}(\Omega))',H^{1}(\Omega)}, 
\end{align}
where $c(\tilde{\pi})=\frac{1}{|\Gamma|}\int_{\Gamma}\widetilde{\pi}$, $(\mathbf{f}_{3},f_{4})\in L^{2}(0,T;\mathbf{V}^{0}(\Omega))\times L^{2}(0,T;L^{2}(\Omega)),$ and $(\mathbf{r},s)$ is the solution to the equation
 \begin{equation}\label{rs}
\left. \begin{aligned}
& -\frac{\partial \mathbf{r}}{\partial t} - \Delta \mathbf{r} + \nabla \tilde{\pi}= \mathbf{f}_{3} \mbox{ in } Q  \\ & \mbox{div }\mathbf{r}=0 \mbox{ in }Q;\quad \mathbf{r}=0 \mbox{ on } \Sigma  \\ &\mathbf{r}(T)=0 \mbox{ in }\Omega\\& -\frac{\partial s}{\partial t}- \Delta s-\mathbf{\beta.r} =f_{4} \mbox{ in }Q  \\& \frac{\partial s}{\partial n} = 0 \mbox{ on } \Sigma;\quad s(T)=0 \mbox{ in }\Omega
 \end{aligned}
\right\}
\end{equation}
and 
\begin{align}\label{pressure}
&\int_{0}^{T}\langle k,p \rangle_{H^{1}(\Omega), \mathcal{H}^{1}(\Omega)}=\int_{0}^{T}\langle \frac{\partial \mathbf{r}_{1}}{\partial \mathbf{n}}(t)- \pi_{1}(t)\mathbf{n} + c(\pi_{1})\mathbf{n}, \mathbf{g}(t) \rangle_{\mathbf{V}^{1/2}(\Gamma),\mathbf{V}^{-1/2}(\Gamma)}\notag\\ &- \int_{0}^{T}\langle  s_{1}(t),h(t) \rangle_{H^{3/2}(\Gamma),H^{-3/2}(\Gamma)}-\langle \mathbf{u}_{0}, \mathbf{r}_{1}(0)\rangle_{\mathbf{H}^{-1}(\Omega),\mathbf{H}_{0}^{1}(\Omega)}-\langle \phi_{0},s(0) \rangle_{(H^{1}(\Omega))',H^{1}(\Omega)},
\end{align}
for every $k\in L^{2}(0,T; H^{1}(\Omega)) \cap H^{\frac{3}{4}}(0,T; H^{-\frac{1}{2}}(\Omega))$ with $\int_{Q}k=0$ and $(\mathbf{r}_{1},s_{1})$ is the solution to 
  \begin{equation}\label{r1s1}
\left\{ \begin{aligned}
& -\frac{\partial \mathbf{r}_{1}}{\partial t} - \Delta \mathbf{r}_{1} + \nabla \pi_{1}= 0 \mbox{ in } Q  \\ & \mbox{div }\mathbf{r}_{1}=k \mbox{ in }Q;\quad \mathbf{r}_{1}=0 \mbox{ on } \Sigma,\quad \mathbf{r}_{1}(T)=0\mbox{  in }\Omega  \\& -\frac{\partial s_{1}}{\partial t}- \Delta s_{1} -\mathbf{\beta}.\mathbf{r}_{1} =0 \mbox{ in }Q  \\& \frac{\partial s_{1}}{\partial n} = 0 \mbox{ on } \Sigma,\quad s_{1}(T)=0\mbox{  in  }\Omega  
 \end{aligned}
\right.
\end{equation}
\end{definition}
 
 \begin{theorem}
For all $(\mathbf{g},h)\in L^{2}(0,T;\mathbf{V}^{-1/2}(\Gamma))\times L^{2}(0,T;H^{-3/2}(\Gamma))$ and\\ $(\mathbf{u}_{0},\phi_{0})\in \mathbf{V}^{-1}(\Omega) \times (H^{1}(\Omega))'$, equation \eqref{aroundzero} admits a unique transposition solution \\ $(\mathbf{u},p,\phi)\in L^{2}(0,T;\mathbf{V}^{0}(\Omega))\times L^{2}(0,T;\mathcal{H}^{-1}(\Omega))\times L^{2}(0,T;L^{2}(\Omega))$ in the sense of\\ above definition.
\end{theorem}
\begin{proof}
Let $(\mathbf{f}_{3},f_{4})\in L^{2}(0,T;\mathbf{V}^{0}(\Omega))\times L^{2}(0,T;L^{2}(\Omega))$, the solution $(\mathbf{r},s)$ to equation \eqref{rs} belongs  to $\V^{2,1}(Q) \times L^{2}(0,T;H^{2}(\Omega))$. Let us define the operator $\Lambda$ from $L^{2}(0,T;\mathbf{V}^{0}(\Omega))\times L^{2}(0,T;L^{2}(\Omega))$ into the space  $ L^{2}(0,T;\V^{\frac{1}{2}}(\Gamma))\times L^{2}(0,T;H^{\frac{3}{2}}(\Gamma)) \times \V^{1}_{0}(\Omega) \times H^{1}(\Omega)$ by : 
\begin{align*}
\Lambda(\mathbf{f}_{3},f_{4})= \Big( -\frac{\partial \mathbf{r}}{\partial \mathbf{n}}+\widetilde{\pi}\mathbf{n}-c(\widetilde{\pi})\mathbf{n},\, s,\, \mathbf{r}(0),\, s(0)\Big),
\end{align*}
where $(\mathbf{r},s)$ is the solution to equation \eqref{rs}. Now, we can rewrite equation \eqref{transposition} in the form :
\begin{align}\label{Lambda}
\Big\langle \begin{pmatrix}
\mathbf{u} \\ \phi
\end{pmatrix},
\begin{pmatrix}
\mathbf{f}_{3} \\ f_{4}
\end{pmatrix} \Big\rangle_{L^{2}(0,T;\V^{0}(\Omega))\times L^{2}(0,T; L^{2}(\Omega))} = \Big \langle \Lambda 
\begin{pmatrix}
\mathbf{f}_{3} \\ f_{4}
\end{pmatrix},
 \begin{pmatrix}
\mathbf{g} \\ h \\ \mathbf{u}_{0} \\ \phi_{0}
\end{pmatrix} \Big \rangle_{1}\,\, ,
\end{align}
where, $\Big\langle .,. \Big\rangle_{1}$ is the duality product between\\ ${L^{2}(0,T;\V^{\frac{1}{2}}(\Gamma))\times L^{2}(0,T;H^{\frac{3}{2}}(\Gamma)) \times \V^{1}_{0}(\Omega) \times H^{1}(\Omega)}$ and its dual. We can check that $\Lambda$ is a linear, bounded operator. Also $\Lambda^{*}$, the adjoint of $\Lambda$ is  bounded and it is from $L^{2}(0,T;\mathbf{V}^{-\frac{1}{2}}(\Gamma))\times L^{2}(0,T;H^{-\frac{3}{2}}(\Gamma))\times \V^{-1}(\Omega) \times (H^{1}(\Omega))'$ into $L^{2}(0,T;\mathbf{V}^{0}(\Omega))\times L^{2}(0,T;L^{2}(\Omega))$.
Thus we have :
 \begin{align}\label{Lambdaadjoint}
 \Big \langle \Lambda 
\begin{pmatrix}
\mathbf{f}_{3} \\ f_{4}
\end{pmatrix},
 \begin{pmatrix}
\mathbf{g} \\ h \\ \mathbf{u}_{0} \\ \phi_{0}
\end{pmatrix} \Big \rangle_{1}= \Big \langle
\begin{pmatrix}
\mathbf{f}_{3} \\ f_{4}
\end{pmatrix}, \Lambda^{*}
 \begin{pmatrix}
\mathbf{g} \\ h \\ \mathbf{u}_{0} \\ \phi_{0}
\end{pmatrix} \Big \rangle_{L^{2}(0,T;\mathbf{V}^{0}(\Omega))\times L^{2}(0,T;L^{2}(\Omega))}.
 \end{align}
 From \eqref{Lambda} and \eqref{Lambdaadjoint}, it is clear that the function 
 \begin{align*}
 \begin{pmatrix}
\mathbf{u} \\ \phi
\end{pmatrix}=\Lambda^{*}\begin{pmatrix}
\mathbf{g} \\ h \\ \mathbf{u}_{0} \\ \phi_{0}
\end{pmatrix}
\end{align*}
is a solution to equation \eqref{aroundzero} in the sense of definition \ref{sa}. 

Existence of pressure $p$ can also be proved in the same way as above. For $k\in L^{2}(0,T;\mathcal{H}^{1}(\Omega)) $ with $\int_{Q}k=0$, the solution $(\mathbf{r}_{1},s_{1})$ to equation \eqref{r1s1} belongs to $\V^{2,1}(Q) \times L^{2}(0,T;H^{2}(\Omega))$. The operator $\Lambda_{1} :$
\begin{align*}
L^{2}(0,T;\mathcal{H}^{1}(\Omega))\mapsto L^{2}(0,T;\V^{\frac{1}{2}}(\Gamma))\times L^{2}(0,T;H^{\frac{3}{2}}(\Gamma)) \times \V^{1}_{0}(\Omega) \times H^{1}(\Omega)
\end{align*}
can be defined as :
\begin{align*}
\Lambda_{1}(k)= \Big( -\frac{\partial \mathbf{r}_{1}}{\partial \mathbf{n}}+\pi_{1}\mathbf{n}-c(\pi_{1})\mathbf{n},\, s_{1},\, \mathbf{r}_{1}(0),\, s_{1}(0)\Big),
\end{align*}
where $(\mathbf{r}_{1},s_{1})$ is the solution to equation \eqref{r1s1}. It can be proved similarly that the function 
\begin{align*}
p= \Lambda_{1}^{*}\begin{pmatrix}
\mathbf{g} \\ h \\ \mathbf{u}_{0} \\ \phi_{0}
\end{pmatrix} \in L^{2}(0,T;\mathcal{H}^{-1}(\Omega))
\end{align*}
is a solution to equation \eqref{aroundzero} in the sense of definition \ref{sa}.\\

To prove uniqueness, let $(\bu,\phi)$ be a solution corresponds to $(\mathbf{g},h,\mathbf{u}_{0},\phi_{0})=(0,0,0,0)$. Then for all $(\mathbf{f}_{3}, f_{4}) \in L^{2}(0,T; \V^{0}(\Omega) \times L^{2}(\Omega))$, by using equation \eqref{transposition} we will get :
\begin{align*}
 \langle (\mathbf{u},\phi),(\mathbf{f}_{3},f_{4}) \rangle_{L^{2}(0,T; \mathbf{V}^{0}(\Omega)\times L^{2}(\Omega))}=0 .
\end{align*}
Hence we have $(\bu, \phi)= (0, 0).$
\end{proof}
Now we want to write equation \eqref{aroundzero} in the operator form when boundary data are regular.
\begin{lemma}
Let $(\mathbf{g},h)\in C^{1}(0,T;\mathbf{V}^{3/2}(\Gamma))\times C^{1}(0,T;H^{1/2}(\Gamma))$. Then $(\mathbf{u},\phi)$, the solution of \eqref{aroundzero} satisfies the following equation:

\begin{align}\label{boussioperator1}
\widetilde{P}\mathbf{u_{\phi}}(t)&=e^{t\mathcal{A}}\widetilde{P}\mathbf{u_{\phi}}(0)+\int_{0}^{t}(-\mathcal{A})e^{(t-s)\mathcal{A}}\widetilde{P}L_{0}(\mathbf{g},h)ds
\end{align}
where $\widetilde{P}=
\begin{pmatrix}
P & 0\\ 0 & I
\end{pmatrix},\,\,
\mathbf{u_{\phi}}(t)=
\begin{pmatrix}
\mathbf{u}(t)\\ \phi(t)
\end{pmatrix},\,\, e^{t\mathcal{A}} \mbox{ is the semigroup generated by } \mathcal{A}=
\begin{pmatrix}
P\Delta & A_{2}\\ 0 & \Delta
\end{pmatrix}$ with $A_{2}(\tau)=P(\beta \tau)$  and $(\mathbf{w}, \pi, \psi)=L_{0}(\mathbf{g},h)$ satisfies the following system:
\begin{equation}\label{LDN}
\left. \begin{aligned}
&-\Delta \mathbf{w}(t) + \nabla \pi(t)= \bm {\beta}\psi(t) \mbox{ in } \Omega  \\ & \mbox{div }\mathbf{w}(t)=0 \mbox{ in } \Omega;\quad \mathbf{w}(t)=\mathbf{g}(t) \mbox{ on } \Gamma, \\
 &- \Delta\psi(t)  = 0 \mbox{ in } \Omega \mbox{ and }\,\,  \frac{\partial\psi}{\partial n}(t) = h(t) \mbox{ on } \Gamma.
 \end{aligned}
\right\}
\end{equation}
\end{lemma}
\begin{proof}
Let $(\mathbf{w}, \pi, \psi)=L_{0}(\mathbf{g},h)$ satisfy system \eqref{LDN}. Thus we will have $(\mathbf{w}, \pi, \psi) \in C^{1}(0,T; \V^{2}(\Omega))\times C^{1}(0,T; H^{1}(\Omega)/\R) \times C^{1}(0,T; H^{2}(\Omega))$. Let us define $\mathbf{y}=\mathbf{u-w},q=p-\pi$ and $\tau=\phi-\psi$ and we will get :
\begin{equation}\label{juju}
\left. \begin{aligned}
&\frac{\partial \mathbf{y}}{\partial t} -\Delta \mathbf{y} + \nabla q= \bm{\beta}\tau-\frac{\partial \mathbf{w}}{\partial t} \mbox{ in } Q,  \\ & \mbox{div }\mathbf{y}=0 \mbox{ in } Q;\quad \mathbf{y}=0 \mbox{ on } \Sigma,\\& \mathbf{y}(0)=\mathbf{u}_{0}-\mathbf{w}(0) \mbox{ in }\Omega,\\
&\frac{\partial \tau}{\partial t} - \Delta\tau  = -\frac{\partial \psi}{\partial t} \mbox{ in } Q,  \\
& \frac{\partial \tau}{\partial n} = 0 \mbox{ on } \Sigma,\quad \tau(0)=\phi_{0}-\psi(0)  \mbox{ in } \Omega.
 \end{aligned}
\right\}
\end{equation}
Now observe that $(\by, q, \tau)\in W(0,T;\V^{1}_{0}(\Omega),\V^{-1}(\Omega)) \times L^{2}(0,T; L^{2}(\Omega)) \times L^{2}(0,T;H^{2}(\Omega))$.  Apply Leray projector $P$ in the first equation of \eqref{juju} and by using $P\mathbf{y}(t)=\mathbf{y}(t)$, we will get : 
\begin{align*}
\mathbf{y}'(t)&=A\mathbf{y}(t)+A_{2}\tau(t)-P\mathbf{w}'(t),\\ \mathbf{y}(0)&=P(\mathbf{u}_{0}-\mathbf{w}(0)),
\end{align*}
where $A: D(A)=\V^{2}(\Omega)\cap \V^{1}_{0}(\Omega)\mapsto \V^{0}_{n}(\Omega)$ is defined by : $A\by=P\Delta \by$ and $A_{2}: L^{2}(\Omega) \mapsto \V^{0}_{n}(\Omega)$ is given by $A_{2} \tau=P(\bm{\beta}\tau)$.
So we can rewrite system \eqref{juju} as :
\begin{align*}
\widehat{\mathbf{z}}'(t)&=\mathcal{A}\widehat{\mathbf{z}}(t)-\mathbf{b}'(t),\\ \widehat{\mathbf{z}}(0)&=\widehat{\mathbf{z}}_{0},
\end{align*}
where $$\widehat{\mathbf{z}}(t)=
\begin{pmatrix}
\mathbf{y}(t)\\ \tau(t)
\end{pmatrix}, \mathbf{b}(t)=
\begin{pmatrix}
P\mathbf{w}(t)\\ \psi(t)
\end{pmatrix}$$
Now we can write $\begin{pmatrix}
P\Delta & P(\bm{\beta}\tau)\\ 0 & \Delta
\end{pmatrix}= \begin{pmatrix}
P\Delta & 0\\ 0 & \Delta
\end{pmatrix} + \begin{pmatrix}
0 & A_{2}\\ 0 & \Delta
\end{pmatrix}= \mathcal{A}_{1}+ \mathcal{A}_{2}$. Clearly, $\mathcal{A}_{1}$ will generate analytic semigroup in $\V^{0}_{n}(\Omega) \times L^{2}(\Omega)$. Observe that $\mathcal{A}$ is a bounded perturbation of $\mathcal{A}_{1}$. Hence it will also generate an analytic semigroup.
Denoting the semigroup generated by $\mathcal{A}$ as $e^{t\mathcal{A}}$, the solution can be written as : 
\begin{align*}
\widehat{\mathbf{z}}(t)&=e^{t\mathcal{A}}\widehat{\mathbf{z}}_{0}-\int_{0}^{t} e^{(t-s)\mathcal{A}}\mathbf{b}'(s)ds\\ &=e^{t\mathcal{A}}(\widehat{\mathbf{z}}_{0}+\mathbf{b}(0))-\mathcal{A}\int_{0}^{t}e^{(t-s)\mathcal{A}}\mathbf{b}(s)ds-\mathbf{b}(t).
\end{align*}
So we have :
\begin{align*}
\widehat{\mathbf{z}}(t)+\mathbf{b}(t)=e^{t\mathbf{A}}(\widehat{\mathbf{z}}_{0}+\mathbf{b}(0))-\mathcal{A}\int_{0}^{t}e^{(t-s)\mathcal{A}}\mathbf{b}(s)ds.
\end{align*}
Thus we obtain $$
\begin{pmatrix}
P\mathbf{u}(t)\\ \phi(t)
\end{pmatrix}=e^{t\mathcal{A}}
\begin{pmatrix}
P\mathbf{u}(0)\\ \phi(0)
\end{pmatrix}
-\mathcal{A}\int_{0}^{t}e^{(t-s)\mathcal{A}}
\begin{pmatrix}
P\mathbf{w}(s)\\ \psi(s)
\end{pmatrix} ds.
$$
So, 
\begin{align*}
\widetilde{P}\mathbf{u}_{\phi}(t)=e^{t\mathcal{A}}\widetilde{P}\mathbf{u}_{\phi}(0)+\int_{0}^{t}(-\mathcal{A})e^{(t-s)\mathcal{A}}\widetilde{P}L_{D,N}(\mathbf{g},h)(s)ds.
\end{align*}
\end{proof}
\begin{remark}\label{adjoint and extension}
 If we just differentiate equation \eqref{boussioperator1} formally, then we can get 
\begin{align*}
\wt{P}\mathbf{u_{\phi}}'= \mc{A}\wt{P}\mathbf{u_{\phi}}+(-\mc{A})\wt{P}L_{0}(\mathbf{g},h).
\end{align*}
But  to make sense of the term $\mc{A}\wt{P}\mathbf{u_{\phi}}$, we want to extend the operator $\mc{A}$ to $\wt{\mc{A}}$ such that $\wt{P}\mathbf{u_{\phi}}$ belongs to the domain of $\wt{\mc{A}}$. We can extend the operator $\mc{A}$ to an unbounded operator $\wt{\mc{A}}$ with domain $D(\wt{A})=\V_{n}^{0}(\Omega)\times L^{2}(\Omega)$ in $(D(\mc{A^{*}}))'$. [ See \cite{PRATO} for details ]
\end{remark}
If the boundary data $\bg.\mathbf{n}\neq 0$, then it is not possible to write equation \eqref{aroundzero} in the following operator form :
\begin{align*}
\mathbf{u_{\phi}}'&= \wt{\mc{A}}\mathbf{u_{\phi}}+(-\wt{\mc{A}})L_{0}(\mathbf{g},h) \mbox{  with  }\mathbf{u_{\phi}}(0)=\begin{pmatrix}
\bu_{0} \\ \phi_{0}
\end{pmatrix},
\end{align*}
as we look for solution $\mathbf{u_{\phi}}$ in the space $\V^{0}(\Omega)\times L^{2}(\Omega)$ but the operator $\wt{\mc{A}}$ is defined on $\V_{n}^{0}(\Omega)\times L^{2}(\Omega)$. To overcome this difficulty we split $\bu_{\phi}=\wt{P}\bu_{\phi} + (I-\wt{P})\bu_{\phi}$, where an evolution equation is satisfied by $\wt{P}\bu$ and a quasi-stationary equation is satisfied by $(I-\wt{P})\bu_{\phi}$.\\
 
 Now we are in a position to state a new definition of weak solution involving $\wt{P}\bu_{\phi}$ and $(I-\wt{P})\bu_{\phi}$ :
\begin{definition}\label{weak solution}
Assume that $(\mathbf{g},h)\in L^{2}(0,T;\mathbf{V}^{0}(\Gamma)) \times L^{2}(0,T;(H^{1}(\Gamma))')$ and $(P\mathbf{u}_{0},\phi_{0})\in \mathbf{V}_{n}^{0}(\Omega)\times L^{2}(\Omega)$.  A function $(\mathbf{u},\phi) \in L^2(0,T;V^{0}(\Omega)\times L^{2}(\Omega))$ is a \textbf{weak solution} to equation \eqref{aroundzero} if $ \wt{P}\mathbf{u_{\phi}}$ is a weak solution of the following evolution system :
\begin{align}
\wt{P}\mathbf{u_{\phi}}'&= \wt{\mc{A}}\wt{P}\mathbf{u_{\phi}}+(-\wt{\mc{A}})\wt{P}L_{0}(\mathbf{g},h)
,\mbox{   with   } \wt{P}\mathbf{u_{\phi}}(0)=\wt{P}\mathbf{u_{\phi}}^{0}, \label{pu01}\\ \mbox{ and   }
 (I-\wt{P})\mathbf{u_{\phi}}(.)&=(I-\wt{P})L_{0}(\mathbf{g}(.),h(.)).  \label{pu02}
\end{align}
where $\widetilde{\mathcal{A}}$ is extension of the operator $\mathcal{A}$ as discussed in remark \ref{adjoint and extension}.
\end{definition}
 Now by definition of weak solution of an evolution equation, a function \\ $\wt{P}\mathbf{u_{\phi}}\in L^2(0,T;V_{n}^{0}(\Omega)\times L^{2}(\Omega))$ is a weak solution to \eqref{pu01} iff for all $(\mathbf{\Phi},\psi)\in D(\mc{A}^{*})$, the mapping $t \mapsto \int_{\Omega} (P\mathbf{u}(t)\mathbf{\Phi}+\phi(t)\psi)$  belongs to $H^{1}(0,T)$ and satisfies
\begin{align} 
 \frac{d}{dt}\int_{\Omega}(P\mathbf{u}(t)\mathbf{\Phi}+\phi(t)\psi)=\langle (P\mathbf{u}(t)+\phi(t)\psi),\mc{A}^{*}(\mathbf{\Phi},\psi)\rangle_{\V_{n}^{0}(\Omega)\times L^{2}(\Omega)} \\ \notag + \langle(-\wt{\mc{A})}\wt{P}L_{0}(\mathbf{g}(t),h(t)),(\mathbf{\Phi},\psi) \rangle_{(D(A^{*}))',D(A^{*})} . 
 \end{align}

\begin{theorem}\label{regularity1}
(i) For all $(P\mathbf{u}_{0},\phi_{0})\in \mathbf{V}_{n}^{0}(\Omega)\times L^{2}(\Omega)$ and $(\mathbf{g},h)\in L^{2}(0,T;\mathbf{V}^{0}(\Gamma))\times L^{2}(0,T;(H^{1}(\Gamma))')$, equation \eqref{pu01}-\eqref{pu02} admits a unique solution $(P\mathbf{u},\phi)$ belongs to $\mathbf{V}^{1/2-\epsilon,1/4-\epsilon/2}(Q)\times H^{1/2-\epsilon,1/4-\epsilon/2}(Q)$ for any $\epsilon>0$. This solution $(P\bu,\phi)=\wt{P}\bu_{\phi}$ and $(I-\wt{P})\bu_{\phi} $ satisfies :
\begin{align}\label{estimate1}
&||\wt{P}\mathbf{u}_{\phi}||_{L^2(0,T;\V_{n}^{1/2-\epsilon}(\Omega)\times H^{1/2-\epsilon}(\Omega))}+\notag\\&||\wt{P}\mathbf{u}_{\phi}||_{H^{1/4-\epsilon/2}(0,T;\V^{0}(\Omega)\times L^{2}(\Omega))}+||(\wt{I}-\wt{P})\mathbf{u}_{\phi}||_{L^2(0,T;\V^{1/2}(\Omega)\times H^{1/2}(\Omega))} \notag\\& \leq C\left(||\wt{P}\mathbf{u}_{\phi}(0)||_{\V_{n}^{0}(\Omega)\times L^{2}(\Omega)}+||(\mathbf{g},h)||_{L^2(0,T;\V^{0}(\Gamma)\times (H^{1}(\Gamma))')} \right),\quad \, \forall\, \epsilon > 0 .
\end{align}

(ii)   For all $(P\mathbf{u}_{0},\phi_{0})\in \mathbf{V}^{3/2}(\Omega)\times H^{3/2}(\Omega)$ and $(\mathbf{g},h)\in L^{2}(0,T;\mathbf{V}^{2}(\Gamma))\cap H^{1}(0,T;\mathbf{V}^{0}(\Gamma))\times L^{2}(0,T;H^{1}(\Gamma))\cap H^{1}(0,T;(H^{1}(\Gamma))')$, equation \eqref{pu01}-\eqref{pu02} admits a unique weak solution $(P\mathbf{u},\phi)\in \V^{5/2-\epsilon,5/4-\epsilon/2} \times H^{5/2- \epsilon,5/4-\epsilon/2}$ under the compatibility condition :
\begin{align}\label{compatibility}
\left.\left(\wt{P}\left[(\mathbf{u}_{0},\phi_{0})-L_{0}(\mathbf{g}(0),h(0))\right]\right)\right\vert_{\Gamma}=0.
\end{align}

(iii)  Let $0 \leq s < 1,$ for all $(P\mathbf{u}_{0},\phi_{0})\in \mathbf{V}^{0 \vee (s-1/2)}(\Omega)\times H^{0 \vee (s-1/2)}(\Omega)$ and $(\mathbf{g},h)\in \mathbf{V}^{s,s/2}(\Sigma)\times (L^{2}(0,T;(H^{1-s}(\Gamma))')\cap H^{s/2}(0,T;((H^{1}(\Gamma))')))$, equation \eqref{pu01}-\eqref{pu02} admits a unique weak solution\\ $(P\mathbf{u},\phi)\in \mathbf{V}^{s+1/2-\epsilon,s/2+1/4-\epsilon/2}(Q)\times H^{s+1/2-\epsilon,s/2+1/4-\epsilon/2}(Q)$.

(iv)   Let $1< s \leq 2,$ for all $(P\mathbf{u}_{0},\phi_{0})\in \mathbf{V}^{(s-1/2)}(\Omega)\times H^{(s-1/2)}(\Omega)$ and $(\mathbf{g},h)\in \mathbf{V}^{s,s/2}(\Sigma)\times (L^{2}(0,T;H^{s-1}(\Gamma))\cap H^{s/2}(0,T;((H^{1}(\Gamma))')))$, equation \eqref{pu01}-\eqref{pu02} admits a unique weak solution $(P\mathbf{u},\phi)\in \mathbf{V}^{s+1/2-\epsilon,s/2+1/4-\epsilon/2}(Q)\times H^{s+1/2-\epsilon,s/2+1/4-\epsilon/2}(Q)$ under the compatibility condition \eqref{compatibility}.
\end{theorem}

\begin{proof}
$(i)$  We will follow the technique of proof used in \cite{RAMO} for the Stokes equation. We have :
\begin{equation}\label{duhamel}
\wt{P}\mathbf{u}_{\phi}(t)=e^{t\mc{A}}\wt{P}\mathbf{u}_{\phi}(0)-\mc{A} \int_0^t e^{(t-s)\mc{A}}\wt{P}L_{0}(\mathbf{g},h)(s)ds.
\end{equation}
Our aim is to prove $\wt{P}\mathbf{u}_{\phi}(t)\in \V_{n}^{1/2-\epsilon}(\Omega) \times H^{1/2-\epsilon}(\Omega)$. Let us consider $0<\epsilon \leq \frac{1}{2}$. Now  from \cite[Chapter II.1.6, proposition 6.1]{PRATO}, we know that $D((-\mc{A})^{1/4-\epsilon/4})\simeq \V^{1/2-\epsilon/2}(\Omega)  \times H^{1/2-\epsilon}(\Omega)$. So, it is enough to show $\wt{P}\mathbf{u}_{\phi}(t) \in D((-\mc{A})^{1/4-\epsilon/4})$.

$\wt{P}\circ L_{0}$ is linear and continuous from $\V^{0}(\Gamma)\times (H^{1}(\Gamma))'$ to $\V_{n}^{1/2}(\Omega) \times H^{1/2}(\Omega)$. Thus if $(\mathbf{g},h)\in L^2(0,T;\V^{0}(\Gamma))\times L^{2}(0,T;(H^{1}(\Gamma))')$, then $(-\mc{A})^{1/4-\epsilon/4}\wt{P}L_{0}(\mathbf{g},h) \in L^2(0,T;\V_{n}^{0}(\Omega))\times L^{2}(0,T;L^{2}(\Omega))$. We also have from \eqref{duhamel} :

\begin{align*}
&||(-\mc{A})^{1/4-\epsilon/2}\wt{P}\mathbf{u}_{\phi}(t)||_{\V_{n}^{0}(\Omega)\times L^{2}(\Omega)}\\& \leq ||e^{t\mc{A}}(-\mc{A})^{1/4-\epsilon/2}\wt{P}\mathbf{u}_{\phi}(0)||_{\V_{n}^{0}(\Omega)\times L^{2}(\Omega)}+\\&\int_0^t ||(-\mc{A})^{1-\epsilon/4}e^{(t-s)\mc{A}}||_{\V_{n}^{0}(\Omega)\times L^{2}(\Omega)} ||(-\mc{A})^{1/4-\epsilon/4}\wt{P}L_{0}(\mathbf{g}(s),h(s))||_{\V_{n}^{0}(\Omega)\times L^{2}(\Omega)} ds\\
& \leq Ct^{-1/4+\epsilon/2}||\wt{P}\mathbf{u}_{\phi}(0)||_{\V_{n}^{0}(\Omega)\times L^{2}(\Omega)}+\\&\int_0^t (t-s)^{-1+\epsilon/4}||(-\mc{A})^{1/4-\epsilon/4}\wt{P}L_{0}(\mathbf{g}(s),h(s))||_{\V_{n}^{0}(\Omega)\times L^{2}(\Omega)} ds\\&(\mbox{ By using \cite[Chapter 2.6, theorem 6.13]{PAZY}})\\& \leq  Ct^{-1/4+\epsilon/2}||\wt{P}\mathbf{u}_{\phi}(0)||_{\V_{n}^{0}(\Omega)\times L^{2}(\Gamma)}+\int_0^t (t-s)^{-1+\epsilon/4}C||(\mathbf{g}(s),h(s))||_{\V^{0}(\Gamma)\times (H^{1}(\Gamma))'}\\&(\mbox{By using \cite[Chapter 2,6, theorem 6.10]{PAZY} })\\&\leq Ct^{-1/4+\epsilon/2}||\wt{P}\mathbf{u}_{\phi}(0)||_{\V_{n}^{0}(\Omega)\times L^{2}(\Omega)}+\\&C\left(\int_0^t (t-s)^{-2+\epsilon/2} ds\right)\left(\int_0^t ||(\mathbf{g}(s),h(s))||_{\V^{0}(\Gamma)\times (H^{1}(\Gamma))'}^{2} ds \right)\\&  \leq Ct^{-1/4+\epsilon/2}||\wt{P}\mathbf{u}_{\phi}(0)||_{\V_{n}^{0}(\Omega)\times L^{2}(\Omega)}+C t^{-1+\epsilon/2} ||(\mathbf{g},h)||_{L^2(0,T;\V^{0}(\Gamma)\times (H^{1}(\Gamma))')} .
\end{align*}
Hence : $\wt{P}\mathbf{u}_{\phi}(t) \in D((-\mc{A})^{1/4-\epsilon/4})$. Also,
\begin{align*}
&||\wt{P}\mathbf{u}_{\phi}||_{L^2(0,T;\V^{1/2-\epsilon}(\Omega))\times H^{1/2-\epsilon}(\Omega)}^{2}=\int_0^T||(-\mc{A})^{-(1/4-\epsilon/2)}(-\mc{A})^{1/4-\epsilon/2}\wt{P}\mathbf{u}_{\phi}(t)||_{\V_{n}^{0}(\Omega)\times L^{2}(\Omega)}^2 dt\\&\leq C\int_0^T ||(-\mc{A})^{1/4-\epsilon/2}\wt{P}\mathbf{u}_{\phi}(t)||_{\V_{n}^{0}(\Omega)\times L^{2}(\Omega)}^2\quad(\mbox{By \cite[Chapter 2.6, lemma 6.3]{PAZY}})\\&\leq C \int_0^T \left(t^{-1/2+\epsilon}||\wt{P}\mathbf{u}_{\phi}(0)||_{\V_{n}^{0}(\Omega)\times L^{2}(\Omega)}^{2}+ t^{-2+\epsilon}||(\mathbf{g},h)||_{L^2(0,T;\V^{0}(\Gamma))}^{2} dt \right)\,\\&(\mbox{As }(a+b)^2\leq 2(a^2+b^2))
\end{align*}
Hence we deduce that:
\begin{equation}\label{PUestimate1}
||\wt{P}\mathbf{u}_{\phi}||_{L^2(0,T;V^{1/2-\epsilon}(\Omega)\times H^{1/2-\epsilon}(\Omega))}\leq C\left(||\wt{P}\mathbf{u}_{\phi}(0)||_{V^{0}_{n}(\Omega)\times L^{2}(\Omega)}+||(\mathbf{g},h)||_{L^2(0,T;V^0(\Gamma))\times (H^{1}(\Gamma))'}\right).
\end{equation}
Our next aim is to show that : $\wt{P}\bu_{\phi}\in H^{1/4-\epsilon/2}(0,T;\V^{0}(\Omega)\times L^{2}(\Omega))$. Moreover from equation \eqref{duhamel} after differentiating w.r.t t:
\begin{align*}
\frac{dP\mathbf{u}_{\phi}}{dt}&=\mc{A}e^{t\wt{\mc{A}}}\wt{P}\mathbf{u}_{\phi}(0)-\mc{A}\wt{P}L_{0}(\mathbf{g},h)(t)-\mc{A}\int_0^t\mc{A}e^{(t-s)\mc{A}}\wt{P}L_{0}(\mathbf{g},h)(s) ds\\&=-(-\mc{A})^{3/4+\epsilon/2}e^{t\mc{A}}(-\mc{A})^{1/4-\epsilon/2}\wt{P}\mathbf{u}_{\phi}(0)+(-\mc{A})^{3/4+\epsilon/2}(-\wt{\mc{A}})^{1/4-\epsilon/2}\wt{P}L_{0}(\mathbf{g},h)(t)\\&+(-\mc{A})^{3/4+\epsilon/2}(-\mc{A})^{1/4-\epsilon/2}\int_0^t \mc{A}e^{(t-s)\mc{A}}\wt{P}L_{0}(\mathbf{g},h)(s)ds\\&=(-\mc{A})^{3/4+\epsilon/2}[-(-\mc{A})^{1/4-\epsilon/2}\wt{P}\mathbf{u}_{\phi}(t)+(-\mc{A})^{1/4-\epsilon/2}\wt{P}L_{0}(\mathbf{g},h)(t)]
\end{align*}
That is we have: 
\begin{align*}
(-\mc{A})^{-3/4-\epsilon/2}\wt{P}\mathbf{u}_{\phi}'=[-(-\mc{A})^{1/4-\epsilon/2}\wt{P}\mathbf{u}_{\phi}(t)+(-\mc{A})^{1/4-\epsilon/2}\wt{P}L_{0}(\mathbf{g},h)(t)].
\end{align*}
That is we have: 
\begin{align*}
(-\mc{A})^{-3/4-\epsilon/2}\wt{P}\mathbf{u}_{\phi}'=[-(-\mc{A})^{1/4-\epsilon/2}\wt{P}\mathbf{u}_{\phi}(t)+(-\mc{A})^{1/4-\epsilon/2}\wt{P}L_{0}(\mathbf{g},h)(t)].
\end{align*}
We deduce that:
\begin{align*}
&||\wt{P}\mathbf{u}_{\phi}'||_{L^2(0,T;[D((-\mc{A})^{3/4+\epsilon/2})]')}\\&\leq C\left(||(-\mc{A})^{1/4-\epsilon/2}\wt{P}\mathbf{u}_{\phi}||_{L^2(0,T;\V_{n}^{0}(\Omega)\times L^{2}(\Omega))}+||(-\mc{A})^{1/4-\epsilon/2}\wt{P}L_{0}(\mathbf{g},h)||_{L^2(0,T;\V_{n}^{0}(\Omega)\times (H^{1}(\Gamma))')}\right)\\ &\leq C\left(||\wt{P}\mathbf{u}_{\phi}||_{L^2(0,T;V^{1/2-\epsilon}(\Omega)\times H^{1/2-\epsilon}(\Omega))}+||(\mathbf{g},h)||_{L^2(0,T;\V^{0}(\Gamma)\times (H^{1}(\Gamma))')} \right)
\end{align*}
Hence we have:
\begin{equation}\label{3}
||\wt{P}\mathbf{u}_{\phi}'||_{L^2(0,T;[D((-\mc{A})^{3/4+\epsilon/2})]')}\leq C\left(||\wt{P}\mathbf{u}_{\phi}(0)||_{\V_{n}^{0}(\Omega)\times L^{2}(\Omega)}+ ||(\mathbf{g},h)||_{L^2(0,T;\V^{0}(\Gamma)\times (H^{1}(\Gamma))')}\right).
\end{equation}
By using interpolation theorem \cite[Chapter 1, theorem 9.7]{LM}, we obtain:
\begin{align}\label{PUestimate2}
||\wt{P}\mathbf{u}_{\phi}||_{H^{1/4-\epsilon/2}(0,T;\V_{n}^{0}(\Omega)\times L^{2}(\Omega))}\leq C\left(||\wt{P}\mathbf{u}_{\phi}(0)||_{\V^{0}_{n}(\Omega)\times L^{2}(\Omega)}+||(\mathbf{g},h)||_{L^2(0,T;\V^0(\Gamma)\times (H^{1}(\Gamma))')}\right).
\end{align}
Now let $(\bg,h)\in L^{2}(0,T;\V^0(\Gamma)\times (H^{1}(\Gamma))')$ and we know that the operator $L_{0}$ is linear and continuous from $\V^{0}(\Gamma)\times (H^{1}(\Gamma))'$ to $\V^{1/2}(\Omega)\times H^{1/2}(\Gamma)$. Observe that :
\begin{align}\label{IPUestimate}
||(I-\wt{P})\mathbf{u}_{\phi}||_{L^2(0,T;\V^{1/2}(\Omega)\times H^{1/2}(\Omega))}^2&=\int_0^T ||(I-\wt{P})\mathbf{u}_{\phi}(t)||_{\V^{1/2}(\Omega)\times H^{1/2}(\Omega)}^2 \notag\\&=\int_0^{T} ||(I-\wt{P})L_{0}(\mathbf{g},h)(t)||_{\V^{1/2}(\Omega)\times H^{1/2}(\Omega)}^2 \notag\\&\leq C||(\mathbf{g},h)||_{L^{2}(0,T;\V^{0}(\Gamma)\times (H^{1}(\Gamma))'}^2 .
\end{align}
Thus by the relations \eqref{PUestimate1}, \eqref{PUestimate2} and \eqref{IPUestimate}, we can establish the estimate \eqref{estimate1}.\\

$(ii)$ Let us consider the case when $(\mathbf{g},h)\in V^{2,1}(\Sigma) \times L^{2}(0,T; H^{1}(\Gamma))\cap H^{1}(0,T; (H^{1}(\Gamma))')$. Also it is given that $(P\mathbf{u}_{0},\phi(0)) \in V^{3/2}(\Omega)\times H^{3/2}(\Omega)$ and $\Big(\wt{P}\Big[(\mathbf{u}_{0},\phi_{0})-L_{0}(\mathbf{g}(0),h(0))\Big]\Big)|_{\Gamma}=0
$.  Our aim is to prove that $\wt{P}\mathbf{u}_{\phi}\in \V^{5/2-\epsilon,5/4-\epsilon/2}(Q) \times H^{5/2-\epsilon,5/4-\epsilon/2}(Q)$ for all $\epsilon>0$. We have: 
\begin{align*}
\wt{P}\mathbf{u}_{\phi}(t)&=e^{t\mc{A}}P\mathbf{u}_{\phi}(0)-\mc{A} \int_0^t e^{(t-s)\mc{A}}\wt{P}L_{0}(\mathbf{g},h)(s)ds\\&=e^{t\mc{A}}\wt{P}\mathbf{u}_{\phi}(0)+ \int_0^t \frac{d}{ds}(e^{(t-s)\mc{A}})\wt{P}L_{0}(\mathbf{g},h)(s)ds\\&=e^{t\mc{A}}\wt{P}\mathbf{u}_{\phi}(0) -\int_0^t e^{(t-s)\mc{A}}\wt{P}L_{0}(\mathbf{g}'(s),h'(s))ds +\\&\quad \wt{P}L_{0}(\mathbf{g},h)(t) -e^{t\mc{A}}\wt{P}L_{0}(\mathbf{g},h)(0).
\end{align*}
Thus we obtain :
\begin{align}\label{PUexpression}
&\wt{P}\mathbf{u}_{\phi}(t)=
e^{t\mc{A}}(\wt{P}\mathbf{u}_{\phi}(0)-\wt{P}L_{0}(\mathbf{g},h)(0))+\wt{P}L_{0}(\mathbf{g},h)(t)-\notag\\& \quad \int_0^t e^{(t-s)\mc{A}}\wt{P}L_{0}(\mathbf{g}',h')(s)ds .
\end{align}
At first we want to show that $\wt{P}\mathbf{u}_{\phi}\in L^2(0,T;\V^{5/2-\epsilon}(\Omega)\times H^{5/2-\epsilon}(\Omega))$. Now by \cite[Chapter II.1.6, proposition 6.1]{PRATO}, we have $D((-\mc{A})^{5/4-\epsilon/2})\simeq \V^{5/2-\epsilon}(\Omega)\times H^{5/2-\epsilon}(\Omega)$. So it is enough to show $\wt{P}\mathbf{u}(t)\in D((-\mc{A})^{5/4-\epsilon/2})$. 

Since $(\mathbf{g},h)\in L^2(0,T;V^2(\Gamma))\cap H^1(0,T;V^0(\Gamma)) \times L^{2}(0,T; H^{1}(\Gamma))\cap H^{1}(0,T; (H^{1}(\Gamma))')$, we have:  
\begin{align*}
\wt{P}L_{0}(\mathbf{g},h)\in L^2(0,T;\V^{5/2}(\Omega)\times H^{5/2}(\Omega))\cap H^1(0,T;\V^{1/2}(\Omega)\times H^{1/2}(\Omega)).
\end{align*} Now,
\begin{align*}
&||(-\mc{A})^{5/4-\epsilon/2}\int_0^t e^{(t-s)\mc{A}}\wt{P}L_{0}(\mathbf{g}',h')(s) ds||_{\V_{n}^{0}(\Omega)\times L^{2}(\Omega)}\\&=||\int_0^t (-\mc{A})^{1-\epsilon/4}e^{(t-s)\mc{A}}(-\mc{A})^{1/4-\epsilon/4}\wt{P}L_{0}(\mathbf{g}',h')(s) ds||_{\V_{n}^{0}(\Omega) \times L^{2}(\Omega)}\\&\leq C\int_0^t (t-s)^{-1+\epsilon/4}||(-\mc{A})^{1/4-\epsilon/4}\wt{P}L_{0}(\mathbf{g}',h')(s)||_{\V_{n}^{0}(\Omega)\times L^{2}(\Omega)} ds\\&\leq C\int_0^t (t-s)^{-1+\epsilon/4}||(\mathbf{g}',h')(s)||_{\V^{0}(\Gamma)\times (H^{1}(\Gamma))'} ds \\&=C t^{\epsilon/4}||(\mathbf{g}',h')||_{L^{2}(0,T;\V^0(\Gamma)\times (H^{1}(\Gamma))')} .
\end{align*}
 Also we obtain :
\begin{align*}
&||(-\mc{A})^{5/4-\epsilon/2}e^{t\mc{A}}(\wt{P}\mathbf{u}_{\phi}(0)-\wt{P}L_{0}(\mathbf{g},h)(0))||_{\V_{n}^{0}(\Omega)\times L^{2}(\Omega)}\\&=||(-\mc{A})^{1/2-\epsilon/4}e^{t\mc{A}}(-\mc{A})^{3/4-\epsilon/4}(\wt{P}\mathbf{u}_{\phi}(0)-\wt{P}L_{0}(\mathbf{g},h)(0))||_{\V_{n}^{0}(\Omega)\times L^{2}(\Omega)}\\&\leq C t^{-1/2+\epsilon/4}||(-\mc{A})^{3/4-\epsilon/4}(\wt{P}\mathbf{u}_{\phi}(0)-\wt{P}L_{0}(\mathbf{g},h)(0))||_{\V_{n}^{0}(\Omega)\times L^{2}(\Omega)}\\& \leq Ct^{-1/2+\epsilon/4}||(\wt{P}\mathbf{u}_{\phi}(0)-\wt{P}L_{0}(\mathbf{g},h)(0))||_{\V^{3/2-\epsilon/2}(\Omega)\times H^{3/2-\epsilon/2}}\\& \leq Ct^{-1/2+\epsilon/4}(||\wt{P}\mathbf{u}_{\phi}(0)||_{\V^{3/2}(\Omega)\times H^{3/2}(\Omega)}+
||\wt{P}L_{0}(\mathbf{g},h)(0)||_{\V^{3/2}(\Omega)\times H^{3/2}(\Omega)}) .
\end{align*}
Thus we have deduced that:
\begin{align}\label{2PU1}
&||\wt{P}\mathbf{u}_{\phi}||_{L^{2}(0,T;\V^{5/2-\epsilon}(\Omega)\times H^{5/2-\epsilon})} \notag\\ &\leq C(||\wt{P}\mathbf{u}_{\phi}(0)||_{\V^{3/2}(\Omega)\times H^{3/2}(\Omega)}+||(\mathbf{g},h)||_{\V^{2,1}(\Sigma)\times L^{2}(0,T; H^{1}(\Gamma))\cap H^{1}(0,T; (H^{1}(\Gamma))')}) .
\end{align}
Our next aim is to prove that $\wt{P}\bu_{\phi} \in H^{1}(0,T; \V^{5/4-\epsilon/2}(\Omega))\times H^{1}(0,T; H^{5/4-\epsilon/2}(\Omega)).$\\

Now by differentiating the expression \eqref{PUexpression} of $\wt{P}\bu_{\phi}$, we shall get:
\begin{align}\label{differentiatePU}
\frac{d\wt{P}\mathbf{u}_{\phi}}{dt}=\mc{A}e^{t\mc{A}}(\wt{P}\mathbf{u}_{\phi}(0)-\wt{P}L_{0}(\mathbf{g},h)(0))-\mc{A}\int_0^t e^{(t-s)\mc{A}}\wt{P}L_{0}(\mathbf{g}',h')(s)ds.
\end{align}

 Now $(\mathbf{g}',h')\in L^2(0,T;V^{0}(\Gamma)) \times L^{2}(0,T; (H^{1}(\Gamma))')$ and by applying (i) of theorem \eqref{regularity1} we already have $\wt{P}\mathbf{u}_{\phi} \in H^{1/4-\epsilon/2}(0,T;V^0(\Omega)\times L^{2}(\Omega))$. Now recall that 
 \begin{align*}
 \wt{P}\mathbf{u}_{\phi}(t)=e^{t\mc{A}}\wt{P}\mathbf{u}_{\phi}(0)-\mc{A}\int_0^t e^{(t-s)\mc{A}}\wt{P}L_{0}(\mathbf{g},h)(s)ds .
 \end{align*}
 As $\wt{P}\mathbf{u}_{\phi}(0)\in \V^{3/2}(\Omega) \times H^{3/2}(\Omega)$, we have $e^{t\mc{A}}\wt{P}\mathbf{u}_{\phi}(0) \in H^{1/4}(0,T;\V^{0}(\Omega)\times L^{2}(\Omega))$. So, we deduce that :
\begin{align}\label{differentiateterm2} 
 \mc{A}\int_0^t e^{(t-s)\mc{A}}\wt{P}L_{0}(\mathbf{g}',h')(s) ds \in H^{1/4-\epsilon/2}(0,T;V^0(\Omega)\times L^{2}(\Omega)). 
\end{align} 
Moreover it is given that $\wt{P}(\mathbf{u}_{\phi}(0)-L_{0}(\mathbf{g},h)(0))\in \V^{3/2}(\Omega)\times H^{3/2}(\Omega)$. Now by an isomorphism theorem \cite[Chapter II.3.2, theorem 2.1]{PRATO} with $D(\mc{A})=\V^2(\Omega)\cap \V_{0}^{1}(\Omega) \times H^{2}(\Omega)$ and $H=\V_{n}^{0}(\Omega)\times L^{2}(\Omega),$ we will get for initial condition in $\V^{1}(\Omega)\times H^{1}(\Omega)$, solution is in $L^2(0,T;\V^2(\Omega) \times H^{2}(\Omega))\cap H^1(0,T;L^2(\Omega)\times L^{2}(\Omega))$. \\

 We can apply same result by replacing $D(\mc{A})$ by $D(\mc{A}^2)$ and $H$ by $D(\mc{A})$ and we can conclude that for initial condition in $\V^3(\Omega)$, solution belongs to $L^2(0,T;\V^4(\Omega) \times H^{4}(\Omega))\cap H^1(0,T;\V^2(\Omega)\times H^{2}(\Omega))$ .\\

 Then by interpolation theorem \cite[Chapter 1, theorem 9.7]{LM}, we can conclude that for $(\wt{P}\mathbf{u}_{\phi}(0)-\wt{P}L_{0}(\mathbf{g},h)(0))\in \V^{3/2}(\Omega)\times H^{3/2}(\Omega)$, we have $e^{t\mc{A}}(\wt{P}\mathbf{u}_{\phi}(0)-\wt{P}L_{0}(\mathbf{g},h)(0)) \in L^2(0,T;\V^{5/2}(\Omega)\times H^{5/2}(\Omega))\cap H^{1}(0,T;\V^{1/2}(\Omega)\times H^{1/2}(\Omega))\hookrightarrow H^{1/4}(0,T;\V^2(\Omega)\times H^{2}(\Omega)).$Thus 
 \begin{align}\label{differentiateterm1}
 \mc{A}e^{t\mc{A}}(\wt{P}\mathbf{u}_{\phi}(0)-\wt{P}L_{0}(\mathbf{g},h)(0)) \in H^{1/4}(0,T; \V^{0}(\Omega)\times L^{2}(\Omega)).
 \end{align} 
 Therefore, by using the relations \eqref{differentiatePU}, \eqref{differentiateterm2} and \eqref{differentiateterm1}, we can conclude that :
 \begin{align}\label{2PU2}
 \frac{d\wt{P}\mathbf{u}_{\phi}}{dt} \in H^{1/4-\epsilon/2}(0,T;\V^{0}(\Omega)\times L^{2}(\Omega)) .
 \end{align}
 It is clear from \eqref{2PU1} and \eqref{2PU2} that : $(P\mathbf{u},\phi)\in \V^{5/2-\epsilon,5/4-\epsilon/2} \times H^{5/2- \epsilon,5/4-\epsilon/2}$ . \\
 
 $(iii)$ As we have already proved the results for $s=0$ and $s=2$, we can get our required result by interpolation.
\end{proof}
We would like to answer the question if we can take $\epsilon=0$ in the results of theorem \ref{regularity1}. We give a complete answer to this question in the next theorem, following the argument of \cite{GRUBBS} as in the case of Navier-Stokes equation.
\begin{theorem}
Assume that $(P\mathbf{u}_{0},\phi_{0})\in \mathbf{V}^{0 \vee (s-1/2)}(\Omega)\times H^{0 \vee (s-1/2)}(\Omega)$, $(\mathbf{g},h)\in \mathbf{V}^{s,s/2}(\Sigma)\times L^{2}(0,T;(H^{1-s}(\Gamma))')\cap H^{s/2}(0,T;(H^{1}(\Gamma))')$,   with $s \in [0,1)$ and $(\mathbf{g},h)\in \mathbf{V}^{s,s/2}(\Sigma)\times L^{2}(0,T;H^{s-1}(\Gamma))\cap H^{s/2}(0,T;(H^{1}(\Gamma))')$ with $s \in (1,2]$. If $(\bu_{0},\phi_{0})$ and $(\bg(0),h(0))$ satisfy the compatibility condition \eqref{compatibility}, then :
\begin{align}
||(P\bu,\phi)||_{\mathbf{V}^{s+1/2,s/2+1/4}(Q)\times H^{s+1/2,s/2+1/4}(Q)} \leq  C(||P\bu_{0}||_{\mathbf{V}^{0 \vee (s-1/2)}(\Omega)}+\notag\\ ||\phi_{0}||_{H^{0 \vee (s-1/2)}(\Omega)}+ ||\bg||_{\mathbf{V}^{s,s/2}(\Sigma)} + ||h||_{L^{2}(0,T;H^{s-1}(\Omega))\cap H^{s/2}(0,T;(H^{1}(\Omega))')}). 
\end{align}
\end{theorem}

\begin{proof}
We can write the first equation of \eqref{aroundzero} as : 
\begin{align}\label{pupde}
\frac{\partial P\bu}{\partial t}- \Delta P\bu + \nabla p= \mathrm{\beta}\phi + \Delta((I-P)\bu)- \frac{\partial (I-P)\bu}{\partial t}.
\end{align}
Let us consider $\bg\in \V^{2,1}(\Sigma)$. Due to theorem \ref{regularity1}, we have $P\bu \in \V^{5/2-\epsilon, 5/4-\epsilon/2}(Q)$ and $(I-P)\bu \in L^{2}(0,T; \V^{5/2}(\Omega)) \cap H^{1}(0,T; \V^{1/2}(\Omega))$, $\phi \in H^{5/2-\epsilon, 5/4-\epsilon/2}(Q)$. Thus from equation \eqref{pupde}, we know the pressure $p \in L^{2}(0,T; H^{1}(\Omega))$. Also, from characterisation of $(I-P)$ we know that $(I-P)\bu= \nabla q$, where $q \in L^{2}(0,T; H^{2}(\Omega)/\R)$ satisfies : 
\begin{align*}
\Delta q(t) = 0 \quad \mbox{in }\Omega,\\ \frac{\partial q}{\partial \mathbf{n}}= \bg(t). \mathbf{n} \quad \mbox{on }\Gamma 
\end{align*}
Thus we can write equation \eqref{pupde} as :
\begin{align}\label{finalpupde}
\frac{\partial P\bu}{\partial t}- \Delta P\bu + \nabla \pi= \mathrm{\beta}\phi ,
\end{align}
where $\pi= p- \Delta q + \frac{\partial q}{\partial t}=p+ \frac{\partial q}{\partial t}L_{z}$. Also observe that :
\begin{align*}
P\bu |_{\Sigma}.\mathbf{n} &= \bu |_{\Sigma}.\mathbf{n} - (I-P)\bu |_{\Sigma}.\mathbf{n}\\ &= \bg.\mathbf{n} - \nabla q |_{\Sigma}.\mathbf{n}\\ &= 0 .
\end{align*}
Hence $P\bu$ satisfies Stokes equation \eqref{finalpupde} with the condition $P\bu |_{\Sigma}.\mathbf{n}=0$.
 Now we are in a position to use the regularity results for instationary Stokes equation with nonhomogeneous boundary condition as in \cite{GRUBBS} to conclude that $P\bu \in \V^{5/2,5/4}(Q)$. Also in \eqref{aroundzero}, $\phi$ satisfies heat equation. So we can use regularity result of heat equation to conclude that $\phi \in H^{5/2,5/4}(Q)$.\\
 Now we can follow same steps as in \cite{RAMO} to get our desired result for $s\in [0,1)\cup (1,2]$. 
\end{proof}
We want to find an appropriate boundary condition for which we can get continuous in time solution for equation \eqref{aroundzero}. The next theorem will give us such a kind of boundary condition :
\begin{corollary}\label{continuity}
If $(\bg,h)\in \V^{3/4,3/4}(\Sigma)\times L^{2}(0,T;(H^{1/4}(\Gamma))') \cap H^{3/4}(0,T;(H^{1}(\Gamma))')$, \\ $(P\bu_{0},\phi_{0}) \in (\V^{3/4}_{n}(\Omega)\times H^{3/4}(\Omega))$ and if the following condition 
\begin{align*}
\wt{P}\Big[(\mathbf{u}_{0},\phi_{0})-L_{0}(\mathbf{g}(0),h(0))\Big]=0,
\end{align*}
holds, then : \\ $(\bu,\phi)\in C([0,T];\V^{3/4}(\Omega))\cap L^{2}(0,T;\V^{5/4}(\Omega)) \times C([0,T];H^{3/4}(\Omega))\cap L^{2}(0,T;H^{5/4}(\Omega))$.
\end{corollary}
\begin{proof}
From definition \ref{weak solution}, we know that : $(I-\wt{P})\mathbf{u_{\phi}}=(I-\wt{P})L_{0}(\mathbf{g},h)$ . As $(\bg,h)\in \V^{3/4,3/4}(\Sigma)\times L^{2}(0,T;(H^{1/4}(\Gamma))') \cap H^{3/4}(0,T;(H^{1}(\Gamma))')$, we will obtain 
\begin{align}\label{IPUspace}
(I-\wt{P})\mathbf{u_{\phi}} \in & L^{2}(0,T; \V^{5/4}(\Omega) \times H^{5/4}(\Omega)) \cap H^{3/4}(0,T; \V^{1/2} \times H^{1/2}(\Omega)). \\ & \hookrightarrow C([0,T]; \V^{3/4}(\Omega)) \times C([0,T]; H^{3/4}(\Omega)). \notag
\end{align}
Our aim is to show : $\wt{P}\bu_{\phi} \in C([0,T]; \V^{3/4}(\Omega)) \times C([0,T]; H^{3/4}(\Omega))$.

When $(\bg,h) \in L^{2}(0,T; \V^{0}(\Gamma)) \times L^{2}(0,T; (H^{1}(\Gamma))')$, from $(i)$ of theorem \eqref{regularity1}, we know that
\begin{align}\label{g00}
 \wt{P}\bu_{\phi} \in \V^{1/2,1/4} \times H^{1/2,1/4}.
\end{align} 
  If $(\bg,h)\in \V^{1.1}(\Sigma) \times L^{2}(0,T; L^{2}(\Gamma)) \cap H^{1}(0,T; (H^{1}(\Gamma))')$ and $\wt{P}\Big[(\mathbf{u}_{0},\phi_{0})-L_{0}(\mathbf{g}(0),h(0))\Big]=0$, then \eqref{PUexpression} gives :
\begin{align}\label{PUspace}
\wt{P}\bu_{\phi}(t)=\wt{P}L_{0}(\mathbf{g},h)(t)- \int_0^t e^{(t-s)\mc{A}}\wt{P}L_{0}(\mathbf{g}',h')(s)ds .
\end{align}
Here 
\begin{align}\label{PUspace2term}
\int_0^t e^{(t-s)\mc{A}}\wt{P}L_{0}(\mathbf{g}',h')(s)ds & \in L^{2}(0,T; \V^{2}(\Omega) \times H^{2}(\Omega)) \cap H^{1}(0,T; \V^{1/4}(\Omega) \times H^{1/4}(\Omega)) \notag \\ & \hookrightarrow H^{2/3}(0,T; \V^{5/6}(\Omega)\times H^{5/6}(\Omega)).
\end{align}
Also we have :
\begin{align}\label{PUspace1term}
\wt{P}L_{0}(\mathbf{g},h) & \in L^{2}(0,T; \V^{3/2}(\Omega)\times H^{3/2}(\Omega)) \cap H^{1}(0,T; \V^{1/2}(\Omega) \times H^{1/2}(\Omega))\notag \\& \hookrightarrow H^{2/3}(0,T; \V^{5/6}(\Omega)\times H^{5/6}(\Omega)).
\end{align}
Thus by using \eqref{PUspace1term} and \eqref{PUspace2term}, the relation \eqref{PUspace} gives : when $(\bg,h)\in \V^{1.1}(\Sigma) \times L^{2}(0,T; L^{2}(\Gamma)) \cap H^{1}(0,T; (H^{1}(\Gamma))')$, 
\begin{align}\label{g11}
\wt{P}\bu_{\phi} \in H^{2/3}(0,T; \V^{5/6}(\Omega)\times H^{5/6}(\Omega)) .
\end{align}
By interpolation between \eqref{g00} and \eqref{g11}, for $(\bg,h)\in \V^{3/4,3/4}(\Sigma)\times L^{2}(0,T;(H^{1/4}(\Gamma))') \cap H^{3/4}(0,T;(H^{1}(\Gamma))')$, we get :
\begin{align}\label{aginPUspace}
\wt{P}\bu_{\phi} &\in C([0,T]; \V^{3/4}(\Omega)) \times C([0,T]; H^{3/4}(\Omega)), \notag \\ \mbox{ and  }\wt{P}\bu_{\phi} & \in L^{2}(0,T; \V^{5/4}(\Omega) \times H^{5/4}(\Omega)).
\end{align}
\end{proof}
\begin{remark}
In the proof of the above corollary, we need the condition $\wt{P}\Big[(\mathbf{u}_{0},\phi_{0})-L_{0}(\mathbf{g}(0),h(0))\Big]=0$. Here, \\ $(\bg,h)\in \V^{3/4,3/4}(\Sigma)\times L^{2}(0,T;(H^{1/4}(\Gamma))') \cap H^{3/4}(0,T;(H^{1}(\Gamma))')$. So,
\begin{align*}
 \wt{P}L_{0}(\mathbf{g},h) &\in L^{2}(0,T; \V_{n}^{5/4}(\Omega) \times H^{5/4}(\Omega))\cap H^{3/4}(0,T; \V_{n}^{1/2}(\Omega)\times H^{1/2}(\Omega)), \\ &\hookrightarrow C([0,T]; \V_{n}^{3/4}(\Omega) \times H^{3/4}(\Omega)).
 \end{align*}
 i.e, $\wt{P}L_{0}(\mathbf{g}(0),h(0)) \in \V_{n}^{3/4}(\Omega)\times H^{3/4}(\Omega).$ That's why we need the condition $(P\bu_{0},\phi_{0})\in \V_{n}^{3/4}(\Omega)\times H^{3/4}(\Omega)$ in corollary \ref{continuity}.
\end{remark}
We can extend the result of last corollary with boundary conditions $\mathbf{g} \in \V^{s,s}(\Sigma)$, $h \in L^{2}(0,T;(H^{1-s}(\Gamma))') \cap H^{s}(0,T;(H^{1}(\Gamma))')$ in the following way :
\begin{corollary}\label{continuity general}
Assume that $\frac{1}{2}< s < 1$. If $(\bg,h)\in \V^{s,s}(\Sigma)\times L^{2}(0,T;(H^{1-s}(\Gamma))') \cap H^{s}(0,T;(H^{1}(\Gamma))')$, \\ $(P\bu_{0},\phi_{0}) \in (\V^{s}_{n}(\Omega)\times H^{s}(\Omega))$ and if the following condition 
\begin{align*}
\wt{P}\Big[(\mathbf{u}_{0},\phi_{0})-L_{0}(\mathbf{g}(0),h(0))\Big]=0,
\end{align*}
holds, then: \\ $(\bu,\phi)\in C([0,T];\V^{s}(\Omega))\cap L^{2}(0,T;\V^{s+\frac{1}{2}}(\Omega)) \times C([0,T];H^{s}(\Omega))\cap L^{2}(0,T;H^{s+\frac{1}{2}}(\Omega))$.
\end{corollary}
\subsection{Regularity of pressure}
 We have a fairly good understanding about the existence of the solution $(\bu,\phi)$ of equation \eqref{aroundzero}. Now we want to understand the space where the pressure lies. We have the following result
 \begin{proposition}
  Assume that $(P\bu_{0},\phi_{0}) \in \V^{s-1/2}_{n}(\Omega)\times H^{s-1/2}(\Omega)$,\\ $(\mathbf{g},h)\in \mathbf{V}^{s,s/2}(\Sigma)\times L^{2}(0,T;H^{s-1}(\Gamma))\cap H^{s/2}(0,T;(H^{1}(\Gamma))')$ with $s\in (1,2]$.\\
 $1.$ If $(P\bu,\phi)\in \mathbf{V}^{s+1/2,s/2+1/4}(Q)\times H^{s+1/2,s/2+1/4}(Q)$ obeys \eqref{pu01}-\eqref{pu02}, then there exists a unique pressure $p \in H^{s/2-1}(0,T; \mathcal{H}^{s-1/2}(\Omega))$ such that $(\bu, p, \phi)$ satisfies \eqref{aroundzero}.
 \end{proposition}  
 We have analysed the space of the solution $(\bu,\phi)$ in the case of $\mathbf{g} \in \V^{s,s}(\Sigma)$, $h \in L^{2}(0,T;(H^{1-s}(\Gamma))') \cap H^{s}(0,T;(H^{1}(\Gamma))')$ [Corollary \ref{continuity} and Corollary \ref{continuity general}]. Now we want to see the space where pressure lies.
 \begin{proposition}
 Assume that $(\mathbf{g},h)\in \mathbf{V}^{s,s}(\Sigma)\times L^{2}(0,T;(H^{1-s}(\Gamma))')\cap\\ H^{s/2}(0,T;(H^{1}(\Gamma))')$ and $(P\bu_{0},\phi_{0}) \in (\V^{s}_{n}(\Omega)\times H^{s}(\Omega))$ with $s\in (\frac{1}{2},1) \cup (1,2]$.\\
 If $(P\bu,\phi)\in \mathbf{V}^{s+1/2,s/2+1/4}(Q)\times H^{s+1/2,s/2+1/4}(Q)$ obeys \eqref{pu01}-\eqref{pu02}, then there exists a unique pressure $p \in H^{-1/2}(0,T; \mathcal{H}^{s-1/2}(\Omega))$ such that $(\bu, p, \phi)$ satisfies \eqref{aroundzero}.
 \end{proposition}
 Now we want to see the regularity of pressure when we provide a better time regularity to the boundary conditions.
  \begin{proposition}
  Assume that $(P\bu_{0},\phi_{0}) \in \V^{s-1/2}_{n}(\Omega)\times H^{s-1/2}(\Omega)$,\\$(\mathbf{g},h)\in L^{2}(0,T; \mathbf{V}^{s}(\Gamma)) \cap H^{\frac{s}{2}+\frac{1}{4}}(0,T; V^{-\frac{1}{2}}(\Gamma))\times L^{2}(0,T;(H^{1-s}(\Gamma))')\cap\\ H^{s/2}(0,T;(H^{\frac{3}{2}}(\Gamma))')$,  with $s\in (1/2,1)$,\\ $(\mathbf{g},h)\in L^{2}(0,T; \mathbf{V}^{s}(\Gamma)) \cap H^{\frac{s}{2}+\frac{1}{4}}(0,T; V^{-\frac{1}{2}}(\Gamma))\times L^{2}(0,T;(H^{s-1}(\Gamma)))\cap\\ H^{s/2}(0,T;(H^{\frac{3}{2}}(\Gamma))')$,  with $s\in (1,2]$ under the compatibility condition \eqref{compatibility}.\\
 If $(P\bu,\phi)\in \mathbf{V}^{s+1/2,s/2+1/4}(Q)\times H^{s+1/2,s/2+1/4}(Q)$ obeys \eqref{pu01}-\eqref{pu02}, then there exists a unique pressure $p \in H^{s/2-3/4}(0,T; \mathcal{H}^{s-1/2}(\Omega))$ such that $(\bu, p, \phi)$ satisfies \eqref{aroundzero}.
 \end{proposition}
 \begin{proof}
 We have $(\bu,\phi) \in \mathbf{V}^{s+1/2,s/2+1/4}(Q)\times H^{s+1/2,s/2+1/4}(Q)$ as in \cite[Lemma 3.2]{RAMO}. We want to determine the correct space for the pressure. Define: 
 \begin{align*}
 \mathbf{U}(\cdot)=\int\limits_{0}^{\cdot} u(t)\, dt \in H^{1}(0,T; \V^{s+\frac{1}{2}}(\Omega)).
 \end{align*} Also we have 
 \begin{align*}
 \mathbf{u}\in L^{2}(0,T; \V^{s+\frac{1}{2}}(\Omega)) \cap H^{\frac{s}{2}+\frac{1}{4}}(0,T; \V^{0}(\Omega)) \hookrightarrow C([0,T]; \V^{s-\frac{1}{2}}(\Omega)).
 \end{align*}We can say that, for all $\mathbf{\phi} \in \V^{0}(\Omega)$, we have
 \begin{align*}
 \left\langle \bu(t)-\bu_{0}-\Delta \mathbf{U}(t), \mathbf{\phi}\right\rangle_{\mathbf{L}^{2}(\Omega)}=0,
 \end{align*} there exists $P(t)$ such that 
 \begin{align*}
 \nabla P(t)=\bu(t)-\bu_{0}-\Delta \mathbf{U}(t) \in \mathbf{V}^{s-\frac{3}{2}}(\Omega).
 \end{align*}Also, we obtain $\nabla P \in H^{\frac{s}{2}+\frac{1}{4}}(0,T; \mathbf{V}^{s-\frac{3}{2}}(\Omega))$. We can define, the pressure function $p=\frac{dP}{dt}$ and we have $p \in H^{\frac{s}{2}-\frac{3}{4}}(0,T; \mathcal{H}^{s-1/2}(\Omega))$.
 \end{proof}
 \begin{remark}
 In \cite{BADRA}, the author assumes that the solution is in the space\\ $L^{2}(0,T;\mathbf{V}^{1}(\Omega))\cap H^{1}(0,T; \mathbf{V}^{-1}(\Omega))$ and also the boundary data has good time regularity. So its not possible to compare our result with \cite[theorem 4.9]{BADRA}.
 \end{remark}
 \subsubsection{Linearization around nonzero stationary state}
 In the previous section, we consider linearization around zero solution. Here we are going to extend the results of previous section to the following linerized system around a stationary state $(\bz, \theta) \in \V^{1}(\Omega) \times H^{1}(\Omega)$ :
   \begin{equation}\label{nonhomogeneoslinearized}
\left. \begin{aligned}
 & \frac{\partial \mathbf{u}}{\partial t} - \Delta \mathbf{u} + (\mathbf{z}.\nabla)\mathbf{u} + (\mathbf{u}.\nabla)\bz + \nabla p= \bm {\beta}\phi  \mbox{ in } Q ;\\ & \mbox{div }\mathbf{u}=0 \mbox{ in } Q;\\ &\mathbf{u}=\mathbf{g} \mbox{ on } \Sigma,\quad \mathbf{u}(0)=\mathbf{u}_{0}(x) \mbox{ in } \Omega; \\& \frac{\partial \phi}{\partial t} - \Delta\phi +(\bz.\nabla)\phi + (\mathbf{u}.\nabla)\theta  = 0 \mbox{ in }\mbox { Q };  \\& \frac{\partial\phi}{\partial n} = h \mbox{ on } \Sigma ,\quad \phi(0)=\phi_{0} \mbox{ in } \Omega.
 \end{aligned}
\right\}
\end{equation}  
  To analyze equation \eqref{nonhomogeneoslinearized}, we will introduce following operators :
  \begin{align*}
  & A_{1}:D(A_{1})\rightarrow \mathbf{V}_{n}^{0}(\Omega); \quad A_{1}\mathbf{u}=P \left(\Delta \mathbf{u}-(\mathbf{u.}\nabla)\mathbf{z}-(\mathbf{z.}\nabla)\mathbf{u} \right),\\
 & A_{2}: L^{2}(\Omega)\rightarrow \mathbf{V}_{n}^{0}(\Omega); \quad A_{2}\phi= P (\bm{\beta}\phi),\\
 & A_{3}:\mathbf{L}^{2}(\Omega)\rightarrow L^{2}(\Omega); \quad  A_{3}\mathbf{u}=-\mathbf{u.}\nabla \theta\\
 & A_{4}: D(A_{4})\rightarrow L^{2}(\Omega); \quad   A_{4}\phi= \Delta \phi -\mathbf{z.}\nabla \phi
  \end{align*}
  where,
  \begin{align*}
  D(A_{1})=\mathbf{V}^{2}(\Omega) \cap \mathbf{V}_{0}^{1}(\Omega) \quad \mbox{and}\quad D(A_{4})=\{\tau \in H^{2}(\Omega)|\,\,\frac{\partial \tau}{\partial n}=0 \}.
  \end{align*}
  Let us define an operator $\mathcal{A}_{\bz,\theta}=
\begin{pmatrix}
A_{1} & A_{2}\\
A_{3} & A_{4}
\end{pmatrix}: D(\mathcal{A}_{\bz,\theta})\mapsto \mathbf{V}_{n}^{0}(\Omega)\times L^{2}(\Omega)$ with $D(\mathcal{A}_{\bz,\theta})=D(A_{1})\times D(A_{4})$. \\

In this section, we will denote by $\lambda_{0} > 0$ an element in the resolvent set of $\mathcal{A}_{\bz,\theta}$ satisfying:
\begin{align}\label{analyticcondition}
\langle (\lambda_{0}I-\mathcal{A}_{\bz,\theta})(\mathbf{\Phi},\psi),(\mathbf{\Phi},\psi)) \rangle_{\mathbf{V}_{n}^{0}(\Omega) \times L^{2}(\Omega)} \geq \omega_{0}|(\mathbf{\Phi},\psi))|_{\mathbf{V}_{0}^{1}(\Omega) \times H_{0}^{1}(\Omega)},
\end{align} 
for some $\omega_{0} > 0$.
\begin{remark}
We can use the regularity of $(\bz, \theta)$ to establish the relation \eqref{analyticcondition}. Use \eqref{analyticcondition} to conclude that $(\lambda_{0}I-\mathcal{A}_{\bz,\theta},D(\mathcal{A}_{\bz,\theta}))$ generates an analytic semigroup on $\mathbf{V}_{n}^{0}(\Omega) \times L^{2}(\Omega)$.\\
Here also we can extend the operator $\mc{A}_{\bz,\theta}$ to an unbounded operator $\wt{\mc{A}_{\bz,\theta}}$ with domain $D(\wt{A}_{\bz,\theta})=\V_{n}^{0}(\Omega)\times L^{2}(\Omega)$ in $(D(\mc{A}_{\bz,\theta}^{*}))'$.
\end{remark}
We can define weak solution as in the case of linearization around zero solution :
\begin{definition}\label{weak solutionz}
Assume that $(\mathbf{g},h)\in L^{2}(0,T;\mathbf{V}^{0}(\Gamma)) \times L^{2}(0,T;(H^{1}(\Gamma))')$ and $(P\mathbf{u}_{0},\phi_{0})\in \mathbf{V}_{n}^{0}(\Omega)\times L^{2}(\Omega)$.  A function $(\mathbf{u},\phi) \in L^2(0,T;V^{0}(\Omega)\times L^{2}(\Omega))$ is a \textbf{weak solution} to equation \eqref{nonhomogeneoslinearized} if $ \wt{P}\mathbf{u_{\phi}}$ is a weak solution of the following evolution system :
\begin{align}\label{puz}
\wt{P}\mathbf{u_{\phi}}'&= \wt{\mc{A}_{\bz,\theta}}\wt{P}\mathbf{u_{\phi}}+(\lambda_{0}I-\wt{\mc{A}_{\bz,\theta}})\wt{P}L_{\bz}(\mathbf{g},h)
,\mbox{   with   } \wt{P}\mathbf{u_{\phi}}(0)=\wt{P}\mathbf{u_{\phi}}^{0},\\ \mbox{ and   }
 (I-\wt{P})\mathbf{u_{\phi}}(.)&=(I-\wt{P})L_{0}(\mathbf{g}(.),h(.)).  \notag
\end{align}
\end{definition}
  \begin{theorem}
  Assume that  $(\bz,\theta) \in \V^{1}(\Omega)\times H^{1}(\Omega)$. For all $(P\mathbf{u}_{0},\phi_{0})\in \mathbf{V}_{n}^{0}(\Omega)\times L^{2}(\Omega)$ and $(\mathbf{g},h)\in L^{2}(0,T;\mathbf{V}^{0}(\Gamma))\times L^{2}(0,T;(H^{1}(\Gamma))')$, equation \eqref{nonhomogeneoslinearized} admits a unique solution $(P\mathbf{u},\phi)\in \mathbf{V}^{1/2-\epsilon,1/4-\epsilon/2}\times H^{1/2-\epsilon,1/4-\epsilon/2}.$
  \end{theorem}
  \begin{proof}
  To prove this we will follow the same steps as in theorem \ref{regularity1} by replacing $\mc{A}, e^{t\mc{A}}, L_{0}$ by $\mc{A}_{\bz,\theta}, e^{t\mc{A}_{\bz,\theta}}, L_{\bz}$.
  \end{proof}
 \subsection{Linearization around an instationary state}
 In this section, our aim is to study the linearized Boussinesq system around an instationary state $(\bz,\theta)$, with homogeneous boundary conditions :
  \begin{equation}\label{homogeneoslinearized}
\left. \begin{aligned}
 & \frac{\partial \mathbf{y}}{\partial t} - \Delta \mathbf{y} + (\mathbf{z}.\nabla)\mathbf{y} + (\mathbf{y}.\nabla)\bz + \nabla p= \bm {\beta}\tau + \mathbf{f}_{1} \mbox{ in } Q ;\\ & \mbox{div }\mathbf{y}=0 \mbox{ in } Q;\\ &\mathbf{y}=\mathbf{0} \mbox{ on } \Sigma, \quad \mathbf{y}(0)=\mathbf{y}_{0}(x) \mbox{ in } \Omega \\& \frac{\partial \tau}{\partial t} - \Delta\tau +(\bz.\nabla)\tau + (\mathbf{y}.\nabla)\theta  = f_{2} \mbox{ in }\mbox { Q }  \\& \frac{\partial\tau}{\partial n} = 0 \mbox{ on } \Sigma ,\quad \tau(0)=\tau_{0} \mbox{ in } \Omega.
 \end{aligned}
\right\}
\end{equation}
Here we assume that $(\bz,\theta)\in L^{2}(0,T;\V^{1}(\Omega)) \cap L^{\infty}(0,T;\mathbf{L}^{4}(\Omega))\times L^{2}(0,T;H^{1}(\Omega)) \cap L^{\infty}(0,T; L^{4}(\Omega))$ and $(\mathbf{f}_{1},f_{2})\in L^{2}(0,T;\mathbf{H}^{-1}(\Omega))\times L^{2}(0,T; H^{-1}(\Omega))$.

We define the operator $\mc{A}_{\bz,\theta}(t)\in \mathcal{L}(\V^{1}_{0}(\Omega)\times H^{1}(\Omega),\V^{-1}(\Omega)\times (H^{1}(\Omega))' )$ by :
\begin{align}\label{bilinearform}
\langle \mc{A}_{\bz,\theta}(t) \begin{pmatrix}
\mathbf{y}\\ \tau
\end{pmatrix},
\begin{pmatrix}
\PHI \\ \psi
\end{pmatrix}
\rangle=
&- \int_{\Omega}\nabla \mathbf{y}.\nabla \mathbf{\Phi}-\int_{\Omega}((\mathbf{z}.\nabla)\mathbf{y}).\mathbf{\Phi}-\int_{\Omega}((\mathbf{y}.\nabla).\bz).\mathbf{\Phi} +\int_{\Omega} \mathbf{\beta}\tau.\mathbf{\Phi} \notag\\&- \int_{\Omega}\nabla \tau.\nabla \psi - \int_{\Omega}(\mathbf{z}.\nabla \tau)\psi - \int_{\Omega} (\mathbf{y}.\nabla \theta)\psi
\end{align}
for all $(\PHI,\psi)\in \V^{1}_{0}(\Omega)\times H^{1}(\Omega)$. We can rewrite equation \eqref{homogeneoslinearized}  in the form :
\begin{align*}
\begin{pmatrix}
\mathbf{y} \\ \tau 
\end{pmatrix}'= \mc{A}_{\bz,\theta}(t)
\begin{pmatrix}
\mathbf{y} \\ \tau 
\end{pmatrix} + \tilde{P}
\begin{pmatrix}
 \mathbf{f}_{1} \\ f_{2}
 \end{pmatrix} \mbox{ with }
 \begin{pmatrix}
 \by(0) \\ \tau(0)
 \end{pmatrix}=
\begin{pmatrix}
 \by_{0} \\ \tau_{0}
 \end{pmatrix} . 
\end{align*}
\begin{lemma}\label{coercive}
Assume that $(\bz,\theta)\in L^{2}(0,T;\V^{1}(\Omega)) \cap L^{\infty}(0,T;\mathbf{L}^{4}(\Omega))\times L^{2}(0,T;H^{1}(\Omega)) \cap L^{\infty}(0,T; L^{4}(\Omega))$. Then there exist $\lambda_{0}>0$ and $M > 0$ such that :
\begin{align}\label{operatorcontinuity}
&(i) \quad |\langle \mc{A}_{\bz,\theta}(t) \begin{pmatrix}
\mathbf{y}\\ \tau
\end{pmatrix},
\begin{pmatrix}
\PHI \\ \psi
\end{pmatrix}
\rangle| \leq M || \begin{pmatrix}
\mathbf{y}\\ \tau
\end{pmatrix}||_{\V_{0}^{1}(\Omega)\times H^{1}(\Omega)}
|| \begin{pmatrix}
\mathbf{\Phi}\\ \psi
\end{pmatrix}||_{\V_{0}^{1}(\Omega)\times H^{1}(\Omega)}
\end{align} and 
\begin{align}\label{operatorcoercive}
&(ii) \quad \langle \lambda_{0}\begin{pmatrix}
\mathbf{y}\\ \tau
\end{pmatrix} - \mc{A}_{\bz,\theta}(t) \begin{pmatrix}
\mathbf{y}\\ \tau
\end{pmatrix},
\begin{pmatrix}
\by \\ \tau
\end{pmatrix}
\rangle \geq \frac{1}{2}\Big(||\by||_{\V^{1}_{0}(\Omega)}^{2}+ ||\tau||^{2}_{H^{1}(\Omega)}\Big).
\end{align}
where $\langle , \rangle = \langle , \rangle_{\V^{-1}(\Omega) \times (H^{1}(\Omega))', \V_{0}^{1}(\Omega)\times H^{1}(\Omega)}$.
\end{lemma}
\begin{proof}
$(i)$ \, For all $(\by,\tau)\in \V^{1}_{0}(\Omega)\times H^{1}(\Omega)$ and all $(\Phi,\psi) \in \V^{1}_{0}(\Omega)\times H^{1}(\Omega)$ and $t \in (0,T)$, we have 
\begin{multline}
|\langle \mc{A}_{\bz,\theta}(t) \begin{pmatrix}
\mathbf{y}\\ \tau
\end{pmatrix},
\begin{pmatrix}
\PHI \\ \psi
\end{pmatrix}
\rangle| \leq  ||\by||_{\V_{0}^{1}(\Omega)}||\Phi||_{\V^{1}_{0}(\Omega)}+ ||\bz||_{L^{\infty}(0,T;\mathbf{L}^{4}(\Omega))}||\by||_{\V^{1}_{0}(\Omega)}||\Phi||_{\mathbf{L}^{4}(\Omega)} \\ + ||\by||_{\mathbf{L}^{4}(\Omega)}||\Phi||_{\V_{0}^{1}(\Omega)}||\bz||_{L^{\infty}(0,T; \mathbf{L}^{4}(\Omega))} + ||\bm{\beta}||_{L^{\infty}}||\tau||_{L^{2}(\Omega)}||\Phi||_{\mathbf{L}^{2}(\Omega)} \\ + ||\tau||_{H^{1}(\Omega)}||\psi||_{H^{1}(\Omega)}+ ||\bz||_{L^{\infty}(0,T;\mathbf{L}^{4}(\Omega))}||\tau||_{H^{1}(\Omega)}||\psi||_{L^{4}(\Omega)}+\\ ||\by||_{\mathbf{L}^{4}(\Omega)}||\psi||_{H^{1}(\Omega)}||\theta||_{L^{\infty}(0,T; L^{4}(\Omega))}  \leq M || \begin{pmatrix}
\mathbf{y}\\ \tau
\end{pmatrix}||_{\V_{0}^{1}(\Omega)\times H^{1}(\Omega)}|| \begin{pmatrix}
\mathbf{\Phi}\\ \psi
\end{pmatrix}||_{\V_{0}^{1}(\Omega)\times H^{1}(\Omega)} .
\end{multline}
Here the last inequality can be obtained by using the fact that $H^{1}(\Omega) \hookrightarrow L^{4}(\Omega)$ is continuous.\\

$(ii)$ \, We can follow the same steps as in the proof of theorem \eqref{existenceboussi} to get our desired estimate.
\end{proof}
\begin{theorem}
Assume that $(\bz,\theta)\in L^{2}(0,T;\V^{1}(\Omega)) \cap L^{\infty}(0,T;\mathbf{L}^{4}(\Omega)) \times \\ L^{2}(0,T;H^{1}(\Omega)) \cap L^{\infty}(0,T; L^{4}(\Omega))$. For all $(\by_{0},\tau_{0}) \in \V_{n}^{0}(\Omega)\times L^{2}(\Omega)$ and $(\mathbf{f}_{1},f_{2})\in L^{2}(0,T;\mathbf{H}^{-1}(\Omega))\times L^{2}(0,T; H^{-1}(\Omega))$, equation \eqref{homogeneoslinearized} admits a unique weak solution\\ $(\by, p, \tau)\in W(0,T;\V^{1}_{0}(\Omega),\V^{-1}(\Omega))\times L^{2}(0,T; L^{2}(\Omega)) \times W(0,T; H^{1}(\Omega),(H^{1}(\Omega))').$ 
\end{theorem}
\begin{proof}
Here $(\by_{0},\tau_{0}) \in \V_{n}^{0}(\Omega)\times L^{2}(\Omega)$ and $(\mathbf{f}_{1},f_{2})\in L^{2}(0,T;\mathbf{H}^{-1}(\Omega))\times L^{2}(0,T; H^{-1}(\Omega))$. Observe that : $\V^{1}_{0}(\Omega)\times H^{1}(\Omega)\hookrightarrow \V_{n}^{0}(\Omega)\times L^{2}(\Omega) \hookrightarrow \V^{-1}(\Omega)\times (H^{1}(\Omega))'$. \\ From \eqref{operatorcontinuity}, we have $\mc{A}_{\bz,\theta}(t)$ is a continuous linear operator from $\V^{1}_{0}(\Omega)\times H^{1}(\Omega)$ to $\V^{-1}(\Omega)\times (H^{1}(\Omega))' )$. Also coercivity of $\mc{A}_{\bz,\theta}(t)$ follows from \eqref{operatorcoercive}.
Then the theorem is a direct consequence of lemma \ref{coercive} and \cite[Chapter 18, Section 3.2, Theorem 1 and 2]{DL}.
\end{proof}
 \section{Nonlinear Boussinesq system}
 In this section , we want to study non homogeneous boussinesq system \eqref{prothomeq}. Let us study at first following homogeneous system: 
 \begin{equation}\label{homogeneousnonlinear}
\left. \begin{aligned}
 & \frac{\partial \mathbf{y}}{\partial t} - \Delta \mathbf{y} + (\mathbf{y}.\nabla)\mathbf{y} + \nabla q= \bm {\beta}\tau + \mathbf{f}_{1} \mbox{ in } Q \\ & \mbox{div }\mathbf{y}=0 \mbox{ in } Q;\\ &\mathbf{y}=\mathbf{0} \mbox{ on } \Sigma,\quad \mathbf{y}(0)=\mathbf{y}_{0}(x) \mbox{ in } \Omega \\& \frac{\partial \tau}{\partial t} - \Delta\tau + (\mathbf{y}.\nabla)\tau  = f_{2} \mbox{ in }\mbox { Q }  \\& \frac{\partial\tau}{\partial n} = 0 \mbox{ on } \Sigma , \tau(0)=\tau_{0} \mbox{ in } \Omega.
 \end{aligned}
\right\}
\end{equation}
\begin{remark}
\quad Equation \eqref{homogeneousnonlinear} can be written as :
\begin{equation}\label{operatorform1}
\left. \begin{aligned}
&\begin{pmatrix}
\mathbf{y}'\\ \tau'
\end{pmatrix}= \mc{A}\begin{pmatrix}
\mathbf{y}\\ \tau
\end{pmatrix}-\wt{P} 
\begin{pmatrix}
(\by.\nabla)\by \\ (\by.\nabla)\tau
\end{pmatrix}-\wt{P}
\begin{pmatrix}
\mathbf{f}_{1} \\ f_{2}
\end{pmatrix},\\ &\begin{pmatrix}
\mathbf{y}\\ \tau
\end{pmatrix}(0)=\begin{pmatrix}
\mathbf{y}_{0}\\ \tau_{0}
\end{pmatrix} .
 \end{aligned}
\right\}
\end{equation}
where, $\mathcal{A}=
\begin{pmatrix}
P\Delta & P(\bm{\beta}\tau)\\ 0 & \Delta
\end{pmatrix}$ .
\end{remark}
\begin{theorem}\label{nonlinear1}
For given $(\mathbf{f}_{1},f_{2})\in L^{2}(0,T;(\V^{1}_{0}(\Omega))')\times L^{2}(0,T; (H^{1}(\Omega))')$ and $(\by_{0},\tau_{0})\in \V_{n}^{0}(\Omega)\times L^{2}(\Omega),$ then system \eqref{homogeneousnonlinear} has at least one weak solution $(\by, q, \tau)$ belongs to $ L^{2}(0,T; \V^{1}_{0}(\Omega))\times L^{2}(0,T; L^{2}(\Omega)) \times L^{2}(0,T; H^{1}(\Omega)).$ Moreover, $(\by,\tau)$ belongs to $L^{\infty}(0,T;\V_{n}^{0}(\Omega)\times L^{2}(\Omega))$ and $(\by,\tau)$ is weakly continuous from $[0,T]$ into $\V_{n}^{0}(\Omega) \times L^{2}(\Omega)$.
\end{theorem}
\begin{proof}
 The proof is based on the construction of approximate solution by the Galerkin method. Since $\V^{1}_{0}(\Omega)$ is separable and $\mathcal{V}(\Omega)$ is dense in $\V^{1}_{0}(\Omega)$, there exists sequence of elements of $\mathcal{V}(\Omega)$, which is basis of $\V^{1}_{0}(\Omega)$. Similarly, there exists a sequence of elements of $C_{c}^{\infty}(\Omega)$ which is a basis of $H^{1}(\Omega)$. Then a passage to the limit using suitable a priori estimate for approximate solution and compactness theorems help us to complete the proof. (See e.g. \cite[Chapter 3, Theorem 3.1]{TEMAM} for Navier-Stokes equation). As equation \eqref{homogeneousnonlinear} is with homogeneous boundary data, we can follow the same steps as in \cite{Morimoto} to complete the proof.
\end{proof}
\begin{remark}\label{split}
\quad Our main aim is to solve nonlinear Boussinesq system \eqref{prothomeq}. Now we write $\bz=\bu+\by$ and $\theta=\phi+\tau$, where $(\bz, p, \theta)$ solves \eqref{prothomeq} and $(\bu,\phi)=\bu_{\phi}$ satisfies :
\begin{equation}\label{pu1}
\left. \begin{aligned}
\wt{P}\mathbf{u_{\phi}}&= \wt{\mc{A}}\wt{P}\mathbf{u_{\phi}}+(-\wt{\mc{A}})\wt{P}L_{0}(\mathbf{g},h)
,\quad \wt{P}\bu_{\phi}(0)=\wt{P}L_{0}(\mathbf{g}(0),h(0)),\\ \mbox{ and   }
 &(I-\wt{P})\mathbf{u_{\phi}}(.)=(I-\wt{P})L_{0}(\mathbf{g}(.),h(.)).  
\end{aligned}
\right\}
\end{equation}
 Here $(\by, q, \tau)$ is the solution of :
 \begin{equation}\label{homogeneousfulllinear}
\left. \begin{aligned}
 & \frac{\partial \mathbf{y}}{\partial t} - \Delta \mathbf{y} + (\mathbf{u}.\nabla)\mathbf{y} + (\mathbf{y}.\nabla)\bu + (\by.\nabla)\by+(\bu.\nabla)\bu+\nabla q= \bm {\beta}\tau  \mbox{ in } Q ;\\ & \mathrm{div}\,\mathbf{y}=0 \mbox{ in } Q;\\ &\mathbf{y}=\mathbf{0} \mbox{ on } \Sigma, \\& \frac{\partial \tau}{\partial t} - \Delta\tau +(\bu.\nabla)\tau + (\mathbf{u}.\nabla)\phi+(\by.\nabla)\phi+(\by.\nabla)\tau  = 0 \mbox{ in }\mbox { Q }  \\& \frac{\partial\tau}{\partial n} = 0 \mbox{ on } \Sigma ,\\ &\begin{pmatrix}
\mathbf{y}\\ \tau
\end{pmatrix}(0)=\wt{P}\begin{pmatrix}
\mathbf{z}(0)\\ \theta(0)
\end{pmatrix}-\wt{P}L_{0}\begin{pmatrix}
\mathbf{g}(0)\\ h(0)
\end{pmatrix} \quad\mbox{  in  }\quad\Omega .
 \end{aligned}
\right\}
\end{equation}
\end{remark}
\begin{lemma}\label{knownterm}
Let $\bu_{\phi}$ be the solution of \eqref{pu1} with $(\bg,h)\in \V^{3/4,3/4}(\Sigma)\times\\C([0,T]) L^{2}(0,T;(H^{1/4}(\Gamma))') \cap H^{3/4}(0,T;(H^{1}(\Gamma))')$. Then $(\bu.\nabla)\bu \in L^{2}(0,T;\V^{-1}(\Omega))$ and $(\bu.\nabla)\phi \in L^{2}(0,T;(H^{1}(\Omega))')$.
\end{lemma}
\begin{proof}
As $(\bg,h)\in \V^{3/4,3/4}(\Sigma)\times L^{2}(0,T;(H^{1/4}(\Gamma))') \cap H^{3/4}(0,T;(H^{1}(\Gamma))')$, from corollary \ref{continuity} we know that :
\begin{align*}
(\bu,\phi)\in C([0,T];\V^{3/4}(\Omega))\cap L^{2}(0,T;\V^{5/4}(\Omega)) \times C([0,T];H^{3/4}(\Omega))\cap L^{2}(0,T;H^{5/4}(\Omega)) .
\end{align*}
Consider $\mathbf{w}\in L^{2}(0,T; \V^{1}_{0}(\Omega))$. To show $(\bu.\nabla)\bu \in L^{2}(0,T;\V^{-1}(\Omega))$, it is enough to establish that : 
\begin{align*}
\int_{0}^{T}|\langle (\bu(t).\nabla)\bu(t),\,\mathbf{w}(t) \rangle_{\V^{-1}(\Omega),\V^{1}_{0}(\Omega)}|^{2} < +\infty .
\end{align*}
Now, we have :
\begin{align*}
\|(\bu.\nabla)\bu\|_{L^{2}(0,T; \V^{-1}(\Omega))} &= \sup_{\|\mathbf{w}\|_{L^{2}(0,T;\V^{1}_{0}(\Omega))}=1} \int_{0}^{T} |\int_{\Omega} (\bu(t).\nabla)\bu(t).\mathbf{w}(t)|^{2}
\end{align*}
On the other hand,
\begin{align*}
\int_{0}^{T} |\int_{\Omega} (\bu(t).\nabla)\bu(t).\mathbf{w}(t)|^{2} &=\int_{0}^{T} |\int_{\Omega} (\bu(t).\nabla)\mathbf{w}(t).\bu(t)|^{2}\\ & \leq \int_{0}^{T} \big( \int_{\Omega} |(\bu(t).\nabla)\mathbf{w}(t).\bu(t)| \big)^{2}\\ & \leq \int_{0}^{T} \big(||\bu(t)||_{\mathbf{L}^{4}(\Omega)}||\mathbf{w}(t)||_{\V^{1}_{0}(\Omega)}||\bu(t)||_{\mathbf{L}^{4}(\Omega)} \big)^{2}\\ & \leq \int_{0}^{T} \big(||\bu(t)||_{\mathbf{V}^{3/4}(\Omega)}||\mathbf{w}(t)||_{\V^{1}_{0}(\Omega)}||\bu(t)||_{\mathbf{V}^{3/4}(\Omega)} \big)^{2}\\& \leq ||\bu||^{4}_{C([0,T];\V^{3/4}(\Omega))}||\bw||^{2}_{L^{2}(0,T; \V^{1}_{0}(\Omega))} < +\infty .
\end{align*}
Similar calculations as above also ensures that $(\bu.\nabla)\phi \in L^{2}(0,T;(H^{1}(\Omega))')$.
\end{proof}
\begin{theorem}\label{nonlinearpart}
Let $(P\bz(0),\theta(0))\in \V_{n}^{0}(\Omega)\times L^{2}(\Omega)$ and  $(\bg,h)\in \V^{s,s}(\Sigma)\times \\ L^{2}(0,T;(H^{1-s}(\Gamma))') \cap H^{s}(0,T;(H^{1}(\Gamma))')$. Then equation \eqref{homogeneousfulllinear}
 admits at least one weak solution $(\by, \tau) \in C_{w}([0,T];\V_{n}^{0}(\Omega)\times L^{2}(\Omega)) \cap (L^{2}(0,T;\V^{1}_{0}(\Omega))\times L^{2}(0,T;H^{1}(\Omega)))$ and $q \in \mathcal{D}'(0,T; L^{2}(\Omega))$, where $C_{w}([0,T];\V_{n}^{0}(\Omega)\times L^{2}(\Omega))$ is the subspace of\\ $L^{\infty}(0,T;\V_{n}^{0}(\Omega)\times L^{2}(\Omega))$ which are continuous from $[0,T]$ into $\V_{n}^{0}(\Omega)\times L^{2}(\Omega)$ equipped with its weak topology. 
\end{theorem}
\begin{proof}
We can rewrite equation \eqref{homogeneousfulllinear} satisfied by $(\by,\tau)$ as :
 \begin{equation}\label{operatorform2}
\left. \begin{aligned}
&\begin{pmatrix}
\mathbf{y}'\\ \tau'
\end{pmatrix}= \mc{A}_{\bz,\theta}(t)\begin{pmatrix}
\mathbf{y}\\ \tau
\end{pmatrix}-\wt{P} 
\begin{pmatrix}
(\by.\nabla)\by \\ (\by.\nabla)\tau
\end{pmatrix}-\wt{P}
\begin{pmatrix}
(\bu.\nabla)\bu \\ (\bu.\nabla)\phi
\end{pmatrix},\\ &\begin{pmatrix}
\mathbf{y}\\ \tau
\end{pmatrix}(0)=\wt{P}\begin{pmatrix}
\mathbf{z}(0)\\ \theta(0)
\end{pmatrix}-\wt{P}L_{0}\begin{pmatrix}
\mathbf{g}(0)\\ h(0)
\end{pmatrix} .
 \end{aligned}
\right\}
\end{equation}
Since $(\bg,h)\in \V^{s,s}(\Sigma)\times L^{2}(0,T;(H^{1-s}(\Gamma))') \cap H^{s}(0,T;(H^{1}(\Gamma))')$, according to Corollary \ref{continuity general}, we conclude that $\wt{P}L_{0}\begin{pmatrix}
\mathbf{g}(0)\\ h(0)
\end{pmatrix} \in \V^{s}(\Omega)\times H^{s}(\Omega)$ and  \\ $(\bu,\phi)\in C([0,T];\V^{s}(\Omega))\cap L^{2}(0,T;\V^{s+\frac{1}{2}}(\Omega)) \times C([0,T];H^{s}(\Omega))\cap L^{2}(0,T;H^{s+\frac{1}{2}}(\Omega))$. Thus it is clear that  $\wt{P}
\begin{pmatrix}
(\bu.\nabla)\bu \\ (\bu.\nabla)\phi
\end{pmatrix} \in L^{2}(0,T;\V^{-1}(\Omega)) \times L^{2}(0,T;(H^{1}(\Omega))')$. A function $(\mathbf{y},\tau)\in C_{w}([0,T];\V_{n}^{0}(\Omega)\times L^{2}(\Omega)) \cap (L^{2}(0,T;\V^{1}_{0}(\Omega))\times L^{2}(0,T;H^{1}(\Omega)))$ is a weak solution to \eqref{operatorform2} iff $\begin{pmatrix}
\widehat{\mathbf{y}}(t) \\ \widehat{\tau}(t)
\end{pmatrix}= e^{-\lambda_{0}t}\begin{pmatrix}
\mathbf{y}(t) \\ \tau(t)
\end{pmatrix}$ is a solution to the system 
 \begin{equation}\label{operatorform3}
\left. \begin{aligned}
&\begin{pmatrix}
\widehat{\mathbf{y}}'\\ \widehat{\tau}'
\end{pmatrix}= \mc{A}_{\bz,\theta}(t)\begin{pmatrix}
\widehat{\mathbf{y}}\\ \widehat{\tau}
\end{pmatrix}-\lambda_{0}\begin{pmatrix}
\widehat{\mathbf{y}}\\ \widehat{\tau}
\end{pmatrix}-\wt{P} 
\begin{pmatrix}
(\widehat{\by}.\nabla)\widehat{\by} \\ (\widehat{\by}.\nabla)\widehat{\tau}
\end{pmatrix}-\lambda_{0}\begin{pmatrix}
\widehat{\mathbf{y}}\\ \widehat{\tau}
\end{pmatrix}-\wt{P}
\begin{pmatrix}
(\bu.\nabla)\bu \\ (\bu.\nabla)\phi
\end{pmatrix},\\ &\begin{pmatrix}
\widehat{\mathbf{y}}\\ \widehat{\tau}
\end{pmatrix}(0)=\wt{P}\begin{pmatrix}
\mathbf{z}(0)\\ \theta(0)
\end{pmatrix}-\wt{P}L_{0}\begin{pmatrix}
\mathbf{g}(0)\\ h(0)
\end{pmatrix} .
 \end{aligned}
\right\}
\end{equation}
Thanks to Lemma \ref{coercive}, the existence in $C_{w}([0,T];\V_{n}^{0}(\Omega)\times L^{2}(\Omega)) \cap (L^{2}(0,T;\V^{1}_{0}(\Omega))\times L^{2}(0,T;H^{1}(\Omega)))$ of $(\widehat{\mathbf{y}},\widehat{\tau})$ satisfying the weak formulation of equation \eqref{operatorform3} can be proved as in the case of Navier-Stokes equations (see \cite[Theorem 5.1]{RAMO}).
%
\end{proof}

\begin{proof}[Proof of Theorem \ref{mainthm}]
As we split solution $(\bz, p, \theta)$ in the form $(\bz, p, \theta)=(\bu, p_{1}, \phi)+(\by, q, \tau)$, theorem \ref{mainthm} follows from corollary \ref{continuity general}, remark \ref{split} and theorem \ref{nonlinearpart}.
\end{proof}

\medskip
Received xxxx 20xx; revised xxxx 20xx.
\medskip

\end{document}